\documentclass[12pt]{article}
\usepackage{amssymb, amsmath, amsthm, amscd}
\usepackage[dvips]{graphics}
\usepackage[utf8]{inputenc}
\usepackage[all,cmtip]{xy}
\usepackage{bbm}
\usepackage{enumitem}
\usepackage{setspace}
\usepackage[colorlinks=true,
            linkcolor=blue,
            urlcolor=blue,
            citecolor=blue]{hyperref}

\begin{document}

\newtheorem{lem}{Lemma}[section]
\newtheorem{pro}[lem]{Proposition}
\newtheorem{defi}[lem]{Definition}
\newtheorem{def/not}[lem]{Definition/Notations}
\newtheorem{thm}[lem]{Theorem}
\newtheorem{ques}[lem]{Question}
\newtheorem{cor}[lem]{Corollary}
\newtheorem{rem}[lem]{Remark}
\newtheorem{rqe}[lem]{Remarks}
\newtheorem{exa}[lem]{Example}
\newtheorem{exas}[lem]{Examples}
\newtheorem{obs}[lem]{Observation}
\newtheorem{corcor}[lem]{Corollary of the corollary}
\newtheorem*{ackn}{Acknowledgements}

\newcommand{\C}{\mathbb{C}}
\newcommand{\R}{\mathbb{R}}
\newcommand{\N}{\mathbb{N}}
\newcommand{\Z}{\mathbb{Z}}
\newcommand{\Q}{\mathbb{Q}}
\newcommand{\Proj}{\mathbb{P}}
\newcommand{\Rc}{\mathcal{R}}
\newcommand{\Oc}{\mathcal{O}}
\newcommand{\diff}{\textit{diff}}

\begin{center}
{\Large\bf  THE HESSIAN EQUATION IN QUATERNIONIC SPACE}

\end{center}
\begin{center}
{\large Hichame Amal \footnote{Department of mathematics, Laboratory LaREAMI, Regional Center of trades of education and training, Kenitra Morocco,  hichameamal@hotmail.com},
 Sa\"{\i}d Asserda \footnote{Ibn tofail university, faculty of sciences, department of mathematics, PO 242 Kenitra Morroco, said.asserda@uit.ac.ma},
 Mohamed Barloub\footnote{Ibn tofail university, faculty of sciences, department of mathematics, PO 242 Kenitra Morroco, mohamed.barloub@uit.ac.ma}
}
\end{center}
\noindent{\small{\bf Abstract.}
In this paper, we introduce  $m$-subharmonic functions in quaternionic space $\mathbb{H}^{n}$,  we define the quaternionic Hessian operator and  solve the homogeneous Dirichlet problem for the quaternionic Hessian equation on the unit ball  with continuous boundary  data.
 
\noindent{\small{\bf Keywords.}
 Potential theory in quaternionic space, $m$-subharmonic function, quaternionic Hessian equation, Dirichlet problem.

\noindent{\small{\bf Mathematics Subject Classification} 32U15, 35J60.
\vspace{1ex}
\section*{Introduction}\label{section:introduction}

 The Baston operator $\Delta$  is the first operator of $0$-Cauchy-Fueter complex on quaternionic manifold:
$$ 0 \longrightarrow  C^{\infty}(\Omega, \mathbb{C})\overset \Delta \longrightarrow  C^{\infty}(\Omega, \wedge^{2}\mathbb{C}^{2n}) \overset D \longrightarrow  C^{\infty}(\Omega, \mathbb{C}^{2}\otimes \wedge^{3}\mathbb{C}^{2n})\longrightarrow  \cdots $$
Alesker defined the quaternionic Monge-Amp\`{e}re operator as the $n$-th power of this
operator  when the manifold is flat. 
Motivated by this formula Wan and Wang in \cite{W2} introduced  two first-order differential operators $d_{0}$ and $d_{1}$ which behaves similarly as $\partial,\; \overline{\partial}$ and $\partial\overline{\partial}$  in complex pluripotential theory, and write $\Delta=d_{0}d_{1}.$ Therefore the quaternionic Monge-Amp\`{e}re operator $(\Delta u)^{n}$ has a simpler  explicit expression,
on this observation, some authors established and developed the quaternionic versions of several results in complex pluripotential theory.
 We consider the following Dirichlet problem for quaternionic Hessian equation in $\Omega\subset \mathbb{H}^{n}$
 \begin{equation}\label{000}  \ref{}
\left\{
     \begin{array}{lll}
            (\Delta u)^m\wedge\beta^{n-m}=fdV,\;\;\;1\leq m \leq n\\
        u_{\vert_{ \partial \Omega}}=\varphi.
     \end{array}
   \right.
\end{equation}

For $m=n$, Alesker  proved in \cite{A1}  that (\ref{000}) is solvable  when $\Omega$ is a strictly pseudoconvex domain, $f\in C(\Omega),\; f\geq 0,\; \varphi\in C(\partial \Omega)$ and the solution is  a continuous plurisubharmonic function. For the smooth case, in \cite{A1}
 he proved a result on the existence and the uniqueness of a  smooth plurisubharmonic  solution  of  (\ref{000}) when $\Omega$ is the euclidian ball in $\mathbb{H}^{n}$ and $f\in C^{\infty}(\Omega) ,\;f> 0,\; \varphi \in C^{\infty} (\partial \Omega)$.
Zhu extended this result in \cite{Z}
 when $\Omega$ is a bounded strictly pseudoconvex domain in $\mathbb{H}^{n}$ provided the existence of a subsolution. Wang and Zhang proved the maximality of locally bounded plurisubharmonic solution to (\ref{000}) with smooth boundary when $\Omega$ is an open set of $\mathbb{H}^{n}$ with $f=0$ and $\varphi \in L_{loc}^{\infty}(\Omega)$. In \cite{S}, Sroka solved
the Dirichlet problem when the right hand side $f$ is merely in $L^p$ for any $p>2$, which is the optimal bound. 
 The h\"{o}lder regularity of the solution  was recently proved independently by Boukhari in \cite{BO} and by Kolodziej and Sroka in \cite{KS},  when $\Omega$ is strongly pseudoconvex bounded domain in $\mathbb{H}^{n}$ with smooth boundary when $\varphi \in C^{1,1}(\partial\Omega),\;0\leq f \in L^{p}(\Omega)$ for $p> 2.$

  For $1 \leq m < n$, in the real case, i.e. when $\Omega$ is a bounded domain in $\mathbb{R}^{n}$, Caffarelly, Nirenberg and Spruck proved in \cite{CNS} the existence and the uniqueness of  smooth admissible solution to the real Hessian equation with $f> 0$ and $\varphi\in C^{\infty}(\overline{\Omega})$ under some suitable conditions.
The complex Hessian equation  was first considered by Li in \cite{Li} where he proved in the spirit of \cite{CNS} that the problem is solvable in the smooth category of admissible functions. In \cite{B}, Blocki gived a natural domain of definition for the complex Monge-Ampère operator and solved the degenerate complex Hessian equation.
For  $ f \geq 0,\;f\in L^{p}$, other authors proved the continuity and the H\"{o}lder of the solution to the Dirichlet problem and developed the theory for the complex Hessian operator see \cite{CHi,DK,NG,SA}.

In this paper, we introduce the class $\mathcal{SH}_{m}$ of $m$-subharmonic functions in $\mathbb{H}^{n},$ the  quaternionic Hessian operator and  give some facts about related pluripotential theory.
 We consider then the following Dirichlet problem
\begin{equation}\label{00}
\left\{
     \begin{array}{lll}
       u\in\mathcal{SH}_{m}(B)\cap C(\overline{B}), \\
        (\Delta u)^m\wedge\beta^{n-m}=0,\\
        u=\varphi       \;\;\;\; \hbox{on}\;\;\;\; \partial B,
     \end{array}
   \right.
\end{equation}
where  $ B$ is the unit  ball and  $\beta=\dfrac{1}{8}\Delta(\parallel q \parallel^{2}).$
We show the existence and the uniqueness of the  solution to (\ref{00}).

Recently Liu and Wang\footnote{We are grateful to the referee for drawing our attention to the existence of this article.} in \cite{LW} introduced and studied  independently $m$-subharmonic functions in quaternion space. Our work and that of Liu and Wang complement each other, the study of the  $m$-Lelong number was done in \cite{LW} and in our work we studied the Dirichlet problem in the unit ball of $\mathbb{H}^{n}$ and maximality.

The paper is organized as follows. In section \ref{sec1}, we recall basic facts about quaternionic linear algebra  and the Baston operator. In section \ref{sec2}, as in the complex case we introduce the concepts of  $m$-subharmonic functions in $\mathbb{H}^{n}$, we define the quaternionic Hessian operator, the capacity, etc. In section \ref{sec3},  we establish a global priori estimates and prove the existence and the uniqueness of solutions to the Dirichlet problem (\ref{00}). Finally, as an application we give a characterization of maximal  $m$-subharmonic functions.

 \section{Prelimenaries}\label{sec1}
 \subsection{Quaternionic linear algebra}\label{7777}

 Throughout this article, we consider  $\mathbb{H}^{n}$ as a right $\mathbb{H}$-vector space, i.e., vectors
are multiplied by scalars on the right. The standard theory of vector spaces,
bases, and dimensions works over $\mathbb{H}^{n}$, exactly like in the commutative case.

\begin{defi}
 A quaternion square matrix $A \in \mathcal{M}_{n}(\mathbb{H})$ is said to be
 Hyperhermitian if $ \overline{A}^{\top}=A, $
 where $ \overline{A} $ means the quaternion conjugate of $ A $ and $ \overline{A}^{\top} $ means the transpose of $  \overline{A} ,$
for simplicity we set $ \overline{A}^{\top}=A^{\ast}. $
\end{defi}
\begin{defi}
Let $ M $ be a quaternion matrix, we say that a quaternion number $ \lambda $ is a right eigenvalue of $ M $ if there exists a unique quaternion vector $ x\neq 0 $ such that $ Mx=x\lambda .$
\end{defi}
\begin{pro}[\cite{W3}]\label{pro00} For a real valued $\mathcal{C}^2$ function $u$, the matrix $\Big[\displaystyle\frac{\partial^2u}{\partial q_j\partial\bar{q}_k}\Big]$ is hyperhermitian.
\end{pro}
\begin{thm}[see pages 145, 146 in \cite{J}]\label{pro0}
If $ M $ is a hyperhermitian matrix, then all right eigenvalues of $ M $ are real numbers.
\end{thm}
\begin{pro}
Let $A$ be a matrix of a given hyperhermitian form in a given basis and
 $C$ be a transition matrix from this basis to another one. Then we have $ A^{'}=C^{\ast}AC $    where $ A^{'} $ denotes the matrix of the given form in the new basis.
\end{pro}
\begin{rem}
The matrix $ C^{\ast}AC $ is hyperhermitian for any hyperhermitian matrix $A$ and any matrix $C$.\\
In particular, $ C^{\ast}C $ is always hyperhermitian.
\end{rem}
\begin{defi}
A hyperhermitian matrix is called positive definite (resp semi positive definite) if $ X^{\ast}AX$ is positive for any vector $ X $ (resp non negative) for any non zero vector $ X .$
\end{defi}
Let $ A $ be a hyperhermitian $ n\times n $ matrix. Suppose $ \sigma $ is a permutation of $ \lbrace 1,\ldots,n\rbrace $.
Write $ \sigma $ as a product of dijoint cycles such that each cycle starts with the smallest number. Since disjoint cycles commute, we can write
$$ \sigma=(k_{11} \ldots k_{1j_{1}})(k_{21} \ldots k_{2j_{2}})\ldots(k_{m1} \ldots k_{mj_{m}}) $$ where for each $ i $ we have $ k_{i1}< k_{ij} $ for all $ j>1 $ and $ k_{11}>k_{21} \ldots >k_{m1} $. This expression is unique.

\begin{defi}
Let $ Sgn(\sigma)$ be the parity of   $\sigma. $
The Moore determinant of $ A $ is
\begin{center}
$$ detA=\sum_{\sigma \in S_{n}} Sgn(\sigma)a_{k_{11}k_{12}}...a_{k_{1j_{1}}k_{11}}a_{k_{21}k_{22}}...a_{k_{m}j_{m}k_{m1}}.$$
\end{center}
\end{defi}
\begin{exa}
\begin{enumerate}
\item[1)] Let$ A=diag(\lambda_{1},...,\lambda_{n})$ be a diagonal matrix with real $ \lambda_{i} $'s, then $A$ is hyperhermitian and the Moore determinant is $$ detA=\prod_{i=1}^{n}\lambda_{i} $$
\item[2)] A general hyperhermitian $ 2\times 2 $ matrix $A$ has the form:
\begin{center}
$ A=\begin{pmatrix}
a & q \\
\overline{q} & b
\end{pmatrix}  $
\end{center}
Where $ a,b \in \mathbb{R} $, $ q\in \mathbb{H} $, then $ detA=ab-q\overline{q} .$
\end{enumerate}
\end{exa}
\begin{rem}
\begin{itemize}
\item[i)] The Moore determinant is the best one for hyperhermitian matrix, because it has almost all algebraic and analytic properties of the usual determinant of real symetric and complexe hyperhermitian matrices.
\item[ii)] The moore determinant of a  hyperhermitian matrix $ M $ is real number and equal to the product of all eigenvalues of $ M .$
\end{itemize}

\end{rem}

\begin{lem}
Let $A$ be a non negative (respectively positive) definite hyperhermitian matrix.\\ Then $ detA \geq 0 $
 (respectively $ detA > 0 $).
 \end{lem}
\begin{proof}
See (lemma 1.1.12) in \cite{A1}
\end{proof}

\subsection{Baston operator $\Delta $ }
 First,we  use the well
known embedding of the quaternionic algebra $ \mathbb{H}$ into $ End(\mathbb{C}^{2})$ defined by:
$$\;\; x_{0}+\mathrm{i}x_{1}+\mathrm{j}x_{2}+\mathrm{k}x_{3} \longrightarrow \begin{pmatrix}
x_{0}+\mathrm{i}x_{1} &-x_{2}-\mathrm{i}x_{3}  \\
x_{2}-\mathrm{i}x_{3} & x_{0}-\mathrm{i}x_{1}
\end{pmatrix},$$
 and the conjugate embedding
$$\begin{array}{c}
\;\; \tau : \mathbb{H}^{n} \cong \mathbb{R}^{4n} \hookrightarrow \mathbb{C}^{2n \times 2} \\
(q_{0},\cdots ,q_{n-1})\mapsto (z^{j\alpha})\in \mathbb{C}^{2n \times 2}
\end{array}$$
$q_{l}=x_{4l}+\mathrm{i}x_{4l+1}+\mathrm{j}x_{4l+2}+\mathrm{k}x_{4l+3}\;,l=0,1,\cdots,n-1\;\;\alpha=0,1$ with
  \begin{equation}\label{1}
  \begin{pmatrix}
 z^{00} & z^{01} \\
  z^{10}&z^{11}  \\
 \vdots &\vdots  \\
z^{(2l)0} & z^{(2l)1}  \\
z^{(2l+1)0} & z^{(2l+1)1}  \\
 \vdots &\vdots  \\
  z^{(2n-2)0}& z^{(2n-2)1}  \\
z^{(2n-1)0} & z^{(2n-1)1}
\end{pmatrix}:=
\begin{pmatrix}
x_{0}-\mathrm{i}x_{1} &-x_{2}+\mathrm{i}x_{3} \\
x_{2} + \mathrm{i}x_{3} & x_{0} + \mathrm{i}x_{1} \\
\vdots  & \vdots \\
x_{4l} - \mathrm{i}x_{4l+1} & -x_{4l+2} + \mathrm{i}x_{4l+3} \\
x_{4l+2} + \mathrm{i}x_{4l+3} & x_{4l} + \mathrm{i}x_{4l+1} \\
\vdots & \vdots \\
x_{4n-4} -\mathrm{i}x_{4n-3} & -x_{4n-2} +\mathrm{i}x_{4n-1} \\
x_{4n-2} + \mathrm{i}x_{4n-3} & x_{4n-4} + \mathrm{i}x_{4n-3}
\end{pmatrix}.
\end{equation}
Pulling back to the quaternionic space $ \mathbb{H}^{n} \cong \mathbb{R}^{4n} $ by  (\ref{1}), we define on $\mathbb{R}^{4n}$ first-order
differential operators $ \bigtriangledown_{j\alpha} $ as follows:
\begin{equation}\label{4}
 \begin{pmatrix}
 \bigtriangledown_{00} & \bigtriangledown_{01} \\
  \bigtriangledown_{10}&\bigtriangledown_{11}  \\
 \vdots &\vdots  \\
\bigtriangledown_{(2l)0} & \bigtriangledown_{(2l)1}  \\
\bigtriangledown_{(2l+1)0} & \bigtriangledown_{(2l+1)1}  \\
 \vdots &\vdots  \\
  \bigtriangledown_{(2n-2)0}& \bigtriangledown_{(2n-2)1}  \\
\bigtriangledown_{(2n-1)0} & \bigtriangledown_{(2n-1)1}
\end{pmatrix}:=
\begin{pmatrix}
\partial_{x_{0}}+\mathrm{i}\partial_{x_{1}} &-\partial_{x_{2}}-\mathrm{i}\partial_{x_{3}} \\
\partial_{x_{2}} - \mathrm{i}\partial_{x_{3}} & \partial_{x_{0}} - \mathrm{i}\partial_{x_{1}} \\
\vdots  & \vdots \\
\partial_{x_{4l}} + \mathrm{i}\partial_{x_{4l+1}} & -\partial_{x_{4l+2}} - \mathrm{i}\partial_{x_{4l+3}} \\
\partial_{x_{4l+2}} - \mathrm{i}\partial_{x_{4l+3}} & \partial_{x_{4l}}- \mathrm{i}\partial_{x_{4l+1}} \\
\vdots & \vdots \\
\partial_{x_{4n-4}} +\mathrm{i}\partial_{x_{4n-3}} & -\partial_{x_{4n-2}} -\mathrm{i}\partial_{x_{4n-1}} \\
\partial_{x_{4n-2}} - \mathrm{i}\partial_{x_{4n-1}} & \partial_{x_{4n-4}} - \mathrm{i}\partial_{x_{4n-3}}
\end{pmatrix}
\end{equation}
 The Baston operator is given by the determinants of $(2 \times 2)$-submatrices above.
Let $\wedge^{2k}\mathbb{C}^{2n}$ be the complex exterior algebra generated by $\mathbb{C}^{2n} , 0 \leq k \leq n$.
 Fix a basis $\lbrace \omega^{0},\omega^{1}\cdots,\omega^{2n-1}\rbrace$ of $\mathbb{C}^{2n}$. Let $\Omega$ be a domain in $\mathbb{R}^{4n}$. We define $$d_{0},d_{1} : C_{0}^{\infty}(\Omega,\wedge^{p}\mathbb{C}^{2n}) \longrightarrow C_{0}^{\infty}(\Omega,\wedge^{p+1}\mathbb{C}^{2n})\;\;\mbox{ by}\;$$
$$d_{0}F :=\sum_{k,I}\bigtriangledown_{k0}f_{I}\omega^{k} \wedge\omega^{I}$$
$$d_{1}F :=\sum_{k,I}\bigtriangledown_{k1}f_{I}\omega^{k} \wedge\omega^{I}$$ and
$$\Delta F:=d_{0}d_{1}F$$ for $F=\sum_{I}f_{I}\omega^{I}\in C_{0}^{\infty}(\Omega,\wedge^{p}\mathbb{C}^{2n}),$
 where $ \omega^{I}:= \omega^{i_{1}} \wedge \ldots  \wedge \omega^{i_{p}} $ for the multi-index $I = (i_{1},\ldots,i_{p})$.
The operators $d_{0}$ and $ d_{1}$ depend on the choice of the coordinates $ x_{j}$’s and the basis $\lbrace \omega^{j}\rbrace$.

 It is known (cf.\cite{W2}) that the second operator $D$ in the 0-Cauchy-Fueter complex can be written as $DF:=\left(
                                                          \begin{array}{c}
                                                            d_{0}F \\
                                                            d_{1}F \\
                                                          \end{array}
                                                        \right).$\\
   Although $d_{0},d_{1}$  are not exterior differential, their behavior is similar to exterior differential:
   \begin{lem}
   $d_{0}d_{1}=-d_{1}d_{0}$,  $d_{0}^{2}=d_{1}^{2}=0$; for $F\in C^{\infty}_{0}(\Omega,\wedge^{p}\mathbb{C}^{2n}),$ $G\in C^{\infty}_{0}(\Omega,\wedge^{q}\mathbb{C}^{2n}),$ we have  \begin{eqnarray}\label{eq3}
   d_{\alpha}(F\wedge G)=d_{\alpha}F\wedge G+(-1)^{p}F\wedge d_{\alpha}G, \        \ \alpha=0,1, \ \ d_{0}\Delta=d_{1}\Delta=0
   \end{eqnarray}
   \end{lem}
   We say $F$ is closed if $d_{0}F=d_{1}F=0,$ ie, $DF=0.$ For $u_{1},u_{2},\ldots,u_{n}\in C^{2},$ $\Delta u_{1}\wedge\ldots\wedge\Delta u_{k}$ is closed, with $k=1,\ldots,n.$
   Moreover, it follows easily from (\ref{4}) that $ \Delta u_{1}\wedge\ldots\wedge\Delta u_{n}$ satisfies the following remarkable identities:
   $$\Delta u_{1}\wedge\ldots\wedge\Delta u_{n}=d_{0}(d_{1}u_{1}\wedge\Delta u_{2}\wedge\ldots\wedge\Delta u_{n})=-d_{1}(d_{0}u_{1}\wedge\Delta u_{2}\wedge\ldots\wedge\Delta u_{n})$$
$$=d_{0}d_{1}(u_{1}\wedge\Delta u_{2}\wedge\ldots\wedge\Delta u_{n})=\Delta(u_{1}\wedge\triangle u_{2}\wedge\ldots\wedge\Delta u_{n}).$$
To write down the explicit expression, we define for a function $u\in C^{2},$
$$\Delta_{ij}u:= \frac{1}{2}(\nabla_{i0}\nabla_{j1}u-\nabla_{i1}\nabla_{j0}u).$$
$2\Delta_{ij}$ is the determinant of $(2\times 2)$-matrix of i-th and j-th rows of (\ref{4}).
 Then we can write
 \begin{equation}\label{5}
 \Delta u=\sum_{i,j=0}^{2n-1}\Delta_{ij}u\omega^{i}\wedge\omega^{j},
 \end{equation}
and for $u_{1},\ldots,u_{n}\in C^{2}$
\begin{equation}
\begin{array}{ll}
\Delta u_{1}\wedge\ldots\wedge\Delta u_{n}&=\sum_{i_{1},j_{1},\ldots}\Delta_{i_{1}j_{1}}u_{1}\ldots\Delta_{i_{n}j_{n}}u_{n}\omega^{i_{1}}\wedge\omega^{j_{1}}\ldots\wedge\omega^{i_{n}}\wedge\omega^{j_{n}}\\
&=\sum_{i_{1},j_{1},\ldots}\delta_{01\ldots(2n-1)}^{i_{1},j_{1}\ldots i_{n}j_{n}}\Delta_{i_{1}j_{1}}u_{i_{1}}\ldots\Delta_{i_{n}j_{n}}u_{n}\Omega_{2n},\\
\end{array}
\end{equation}
 where $$\Omega_{2n}=\omega^{0}\wedge \omega^{1}\wedge\ldots\wedge\omega^{2n-1}$$
 and $\delta_{01\ldots(2n-1)}^{i_{1},j_{1}\ldots i_{n}j_{n}}$:= the sign of the permutation from $(i_{1},j_{1},\ldots,i_{n},j_{n})$ to $(0,1,\ldots,2n-1),$
 if
$\{i_{1},j_{1}...,i_{n},j_{n}\}=\{0,1,...,2n-1\};$ otherwise, $\delta_{01..(2n-1)}^{i_{1}j_{1}..i_{n}j_{n}}=0.$ Note that $ \Delta u_{1}\wedge\ldots\wedge\Delta u_{n}$ is symmetric with respect to the permutation of $u_{1},...,u_{n}.$ In particulier, when $u_{1}=...=u_{n}=u,\;  \Delta u_{1}\wedge\ldots\wedge\Delta u_{n}$ coincides with $(\Delta u)^{n}=\wedge^{n}\Delta u.$\\
We denote by $\Delta_{n}( u_{1},..., u_{n})$ the coefficient of the form $ \Delta u_{1}\wedge\ldots\wedge\Delta u_{n},$ ie, $\Delta u_{1}\wedge\ldots\wedge\Delta u_{n}=\Delta_{n}( u_{1},..., u_{n})\Omega_{2n}.$ Then $\Delta_{n}( u_{1},..., u_{n})$ coincides with the mixed Monge-Amp\`{e}re operator $\det( u_{1},..., u_{n})$ while $\Delta_{n}u$ coincides with the quaternionic Monge-Amp\`{e}re operator $\det(u).$ See Appendix A of \cite{W2} for the proof.\\
Following Alesker \cite{A2}, we denote by $\wedge_{\mathbb{R}}^{2k}\mathbb{C}^{2n}$ the subspace of all real elements in $\wedge^{2k}\mathbb{C}^{2n}$ . They are counterparts of $(k,k)-$forms in several complex variables. In the space $\wedge_{\mathbb{R}}^{2k}\mathbb{C}^{2n}$, Wan and Wang in \cite{W2} defined convex cones $\wedge_{\mathbb{R}^{+}}^{2k}\mathbb{C}^{2n}$ and $SP^{2k}\mathbb{C}^{2n}$ of positive and strongly positive elements, respectively. Denoted by $\mathcal{D}^{2k}(\Omega)$ the set of all $C_{0}^{\infty}(\Omega)$ functions valued in $\wedge^{2k}\mathbb{C}^{2n}.$ A form $\eta\in\mathcal{D}^{2k}(\Omega)$ is called  positive (respectively, strongly positive ) if for any $q\in\Omega,$ $\eta(q)$ is a positive (respectively, strongly positive) element. Such forms are the same as the sections of certain line bundle introduced by Alesker \cite{A2} when the manifold is flat. Denote by $PSH$ the class of all quaternionic plurisubharmonic functions. It's proved in \cite[Proposition 3.2]{W2},  that for $u\in PSH\cap C^{2}(\Omega),\;\; \Delta u$ is a closed strongly positive $2$-form.\\
An element of the dual space $(\mathcal{D}^{2n-p}(\Omega))'$ is called a $p$-current. Denoted by $\mathcal{D}_{0}^{p}(\Omega)$ the set of all $C_{0}(\Omega)$ functions valued in $\wedge^{p}\mathbb{C}^{2n}.$ The elements of the dual space $(\mathcal{D}_{0}^{2n-p}(\Omega))'$ are called $p$-currents of order zero. Obviously, the $2n$-currents are just the distributions on $\Omega.$\\
A 2k-current $T$ is said to be positive if we have $T(\eta)\geq0$ for any strongly positive form $\eta\in\mathcal{D}^{2n-2k}(\Omega).$
Although a $2n$-form is not an authentic differential form and we cannot integrate it, we can define
$$\displaystyle{\int}_{\Omega}F:=\displaystyle{\int}_{\Omega}fdV,$$ if we write $F=f\Omega_{2n}\in L^{1}(\Omega,\wedge^{2n}\mathbb{C}^{2n}),$ where $dV$ is the Lebesgue measure.
In particular, if $F$ is positive $2n$-form, then $\displaystyle{\int}_{\Omega}F\geq0.$ For a $2n$-current $F=\mu\Omega_{2n}$ with coefficient to be measure $\mu,$ define
$\displaystyle{\int}_{\Omega}F:=\displaystyle{\int}_{\Omega}\mu.$ Any positive $2k$-current $T$ on $\Omega$ has measure coefficients (cf \cite{W2} for more details). For a positive $2k$-current $T$ and a strongly positive test form $\varphi\in \mathcal{D}^{2n-2k}(\Omega),$ we can write $$T\wedge\varphi=\mu\Omega_{2n}  \;\;\mbox{and}\;\;  T(\varphi)=\displaystyle{\int}_{\Omega}T\wedge\varphi$$ for some Radon measure $\mu.$
Now for the $p$-current $F,$ we define $d_{\alpha}F$ as $(d_{\alpha}F)(\eta):=-F(d_{\alpha}\eta),$ $\alpha=0,1,$ for any test $(2n-p-1)$-form $\eta.$ We say that a current $F$ is closed if $d_{0}F=d_{1}F=0.$ Wan and Wang proved that the current $\Delta u$ is closed positive $2$-current for any $u\in PSH(\Omega)$ (see \cite[Theorem 3.7]{W2}).
\begin{lem}{(Stokes type formula,\cite[Lemma 3.2 ]{W2})}\label{lem1}

Assume that $T$ is a smooth $(2n-1)$-form in $ \Omega$, and
$h$ is a smooth function with $h = 0$ on $\partial \Omega$. Then we have
$\int_{\Omega}hd_{\alpha}T=-\int_{\Omega}d_{\alpha}h\wedge T\;\;\;\;\alpha=0,1.$
\end{lem}
The theory of Bedford-Taylor \cite{BE} in complex analysis can be generalized to the quaternionic case. Let $u$ be a locally bounded PSH function and let $T$ be a closed positive $2k$-current. Define $$\Delta u\wedge T:=\Delta(uT),$$ i.e., $(\Delta u\wedge T)(\eta):=uT(\Delta\eta)$ for test form $\eta.$ $\Delta u\wedge T$ is  also a closed positive current. Inductively, for $u_{1},\ldots,u_{p}\in PSH\cap L_{loc}^{\infty}(\Omega),$ Wan and Wang in \cite{W2} showed that
$$\Delta u_{1}\wedge\ldots\wedge\Delta u_{p}:=\Delta(u_{1}\Delta u_{2}\wedge\ldots\wedge\Delta u_{p})$$ is closed  positive $2p$-current. In particular, for $u_{1},\ldots,u_{n}\in PSH\cap L_{loc}^{\infty}(\Omega),$ $\Delta u_{1}\wedge\ldots\wedge\Delta u_{n}=\mu\Omega_{2n}$ for a well-defined positive Radon measure $\mu.$\\
For any test $(2n-2p)$-form $\psi$ on $\Omega,$ we have $\displaystyle{\int}_{\Omega}\Delta u_{1}\wedge\ldots\wedge\Delta u_{p}\wedge\psi=\displaystyle{\int}_{\Omega}u_{1}\Delta u_{2}\wedge\ldots\wedge\Delta u_{p}\wedge\Delta\psi,$
where $u_{1},\ldots,u_{p}\in PSH\cap L_{loc}^{\infty}(\Omega).$
Given a bounded plurisubharmonic function $u$ one can define the quaternionic Monge-Amp\`{e}re measure
$$ (\Delta u)^{n}=\Delta u\wedge\Delta u\wedge\ldots\wedge\Delta u.$$
This is a nonnegative Borel measure.
\section{The $m$-subharmonic functions in $\mathbb{H}^{n}$}\label{sec2}
\subsection{Basic properties of elementary symmetric functions}
 Let  $1\leq m\leq n,$ we set  $$S_m(\lambda)=\sum_{1\leq j_1<\ldots< j_m\leq n}\lambda_{j_{1}}\ldots\lambda_{j_{m}},$$
called the symetric  function of $\mathbb{R}^{n}$ of degree $m$, which can be determined by
                 $$(\lambda_{1}+t)\ldots(\lambda_{n}+t)=\sum_{m=0}^{n}S_m(\lambda)t^{n-m}, \;\;\;\; \hbox{with}\;\;\; t\in \mathbb{R}.$$
We denote $\Gamma_{m}$ the closure of the  connected component of $\{\lambda\in \mathbb{R}^{n}: S_m(\lambda)>0\}$ containing $(1,\ldots,1).$

 Let $t\geq 0$, we have
 $$ \Gamma_{m}= \lbrace \lambda \in \mathbb{R}^{n}: S_{m}(\lambda_{1}+t,\ldots,\lambda_{n}+t)\geq 0 \rbrace = \lbrace \lambda \in \mathbb{R}^{n}: \sum_{p=0}^{m}S_{p}(\lambda)t^{n-p}\geq 0 \rbrace = \displaystyle\bigcap_{p=0}^{m}\lbrace S_{p} \geq0 \rbrace . $$

Note that, $\Gamma_{n}\subset \Gamma_{n-1}\subset \ldots\subset \Gamma_{1},$ and by the results in \cite{G}, $\Gamma_{m}$ is convex in $\mathbb{R}^{n}$ and $(S_{m})^{\frac{1}{m}}$ is concave in $\Gamma_{m}$, and by the Maclaurin inequality
$$\left(
    \begin{array}{c}
      n \\
      m \\
    \end{array}
  \right)^{\frac{-1}{m}}(S_{m})^{\frac{1}{m}}\leq\left(
    \begin{array}{c}
      n \\
      p \\
    \end{array}
  \right)^{\frac{-1}{p}}(S_{p})^{\frac{1}{p}}, \;\;\; \forall\; 1\leq p\leq m\leq n.
$$ Remind that a
square quaternionic matrix $A = (a_{ij} )$ is called hyperhermitian if its quaternionic conjugate
$A^{\ast} = A$, or explicitly $a_{ij} = \overline{a_{ji}}$.
Let $\mathcal{H}_{n}$ be the real vector space of  hyperhermitian $(n\times n)$-matrices, then from Theorem \ref{pro0} follow  the eigenvalues $\lambda(A)
=( \lambda_{1},\ldots,\lambda_{n})$ are real for any  $A \in \mathcal{H}_{n}$. We set                                       $\tilde{S}_{m}(A)=S_{m}(\lambda(A))$
and define the cone
$$\tilde{\Gamma}_{m}:=\lbrace A \in \mathcal{H}_{n}: \lambda(A)\in \Gamma_{m} \rbrace= \lbrace A \in \mathcal{H}_{n}: \tilde{S}_{k}(A)\geq0, \forall \;1\leq k \leq m \rbrace.$$
Let $ M : \mathcal{H}_{n}^{m} \longrightarrow \mathbb{R}$ be the polarized form of $ \tilde{S}_{m}$, it is determined by
the following three properties:
 $M$ is linear in every variable, symmetric and
$ M(A, \ldots , A) = \tilde{S}_{m}(A),\;\; A \in \mathcal{H}_{n}.$
The inequality due to Garding \cite[Theorem 5 ]{G} and \cite[corollary 2.2]{LW} asserts that
\begin{equation} \label{8}
 M(A_{1}, \ldots , A_{m})\geq \tilde{S}_{m}(A_{1})^{\frac{1}{m}} \ldots \tilde{S}_{m}(A_{m})^{\frac{1}{m}} \;\;, A_{1}, \ldots , A_{m} \in \tilde{\Gamma}_{m}.
\end{equation}
 For $B\in \mathcal{H}_{n}$ we define
$\mathcal{F}_{m}(B):=\Big(\dfrac{\partial\tilde{S}_{m}}{\partial b_{p\overline{q}}}(B)\Big)\in \mathcal{H}_{n}.$ Then
$tr(A\mathcal{F}_{m}(B))=mM(A,B,\ldots,B),$\\
in particular
\begin{equation}\label{81}
 tr(B\mathcal{F}_{m}(B))=\tilde{S}_{m}(B).
\end{equation}
If $B$ is diagonal then so is $\mathcal{F}_{m}(B).$
 If $\lambda =\lambda(B)$ then
$$\lambda(\mathcal{F}_{m}(B)) = (S_{m-1}(\lambda^{(1)}),\ldots,S_{m-1}(\lambda^{(n)})),$$
where $ \lambda^{(j)} = (\lambda_{1},\ldots,\lambda_{j-1},\lambda_{j+1},\ldots,\lambda_{n}).$
By (\ref{8}) we have
\begin{equation}\label{82}
 tr(A\mathcal{F}_{m}(B))\geq m\tilde{S}_{m}(A)^{\frac{1}{m}}\tilde{S}_{m}(B)^{\frac{m-1}{m}}, \;\;A,B\in \tilde{\Gamma}_{m}.
\end{equation}
 \subsection{The $m$-subharmonic functions}
  Let $$\beta=\sum_{l=0}^{n-1}\omega^{2l}\wedge \omega^{2l+1}\in \wedge_{\mathbb{R}^{+}}^{2}\mathbb{C}^{2n} $$ be the standard k\"{a}hler form,
let $\Omega$ be a bounded domain in $\mathbb{H}^{n}$ and let $u \in C^{2}(\Omega)$ be a real valued function, then by Proposition \ref{pro00} the quaternionic Hessian of $u$ denoted by: $(u_{l\overline{k}}):= \Big(\dfrac{\partial^{2}u}{\partial q_{l}\partial\overline{q}_{k}}\Big)_{n \times n}$ is an hyperhermitian matrix and so  $\lambda(u_{l\overline{k}})$ are
real. The quaternionic Hessian operator is defined by:  $$\tilde{S}_{m}(u_{l\overline{k}}):=S_{m}(\lambda(u_{l\overline{k}})) \mbox{ where } \lambda(u_{l\overline{k}})=(\lambda_{1}(u),\ldots,\lambda_{n}(u)) \mbox{ are the eigenvalues of the Hessian }.$$
 \begin{defi}
 A real valued function $u$ in $C^{2}(\Omega)$ where $\Omega\subset \mathbb{H}^{n}$ is called $m$-subharmonic (we write $u \in \mathcal{SH}_{m}$) if the eigenvalues
 $\lambda(u)=(\lambda_{1}(u),\ldots,\lambda_{n}(u))$ of the quaternionic Hessian $(u_{l\overline{k}})$ belong to $\Gamma_{m}$.
 \end{defi}
 Let $$\widehat{\Gamma}_{m}
  :=\lbrace \alpha \in \wedge_{\mathbb{R}}^{2}\mathbb{C}^{2n} \;\; / \alpha \wedge \beta^{n-1} \geq 0,\alpha^{2} \wedge \beta^{n-2} \geq 0,\ldots ,\alpha^{m} \wedge \beta^{n-m} \geq 0 \rbrace, $$
where $\wedge_{\mathbb{R}}^{2}\mathbb{C}^{2n}$ is the space of all real $2$-forms in quaternion analysis.
\begin{pro}
A fonction $u\in C^{2}(\Omega)$ is $m$-subharmonic if and only if $\Delta u \in \widehat{\Gamma}_{m}$.
\end{pro}
\begin{proof}
Let $u\in C^{2}(\Omega)$, by (\ref{4}) and (\ref{5}) we get
$$\begin{array}{ll}
 \dfrac{\partial^{2} u}{\partial q_{l}\partial \overline{q }_{k}}&=(\dfrac{\partial u }{\partial x_{4l}}+\mathrm{i}\dfrac{\partial u}{\partial x_{4l+1}}+\mathrm{j}\dfrac{\partial u}{\partial x_{4l+2}}+\mathrm{k}\dfrac{\partial u}{\partial x_{4l+3}}). (\dfrac{\partial u}{\partial x_{4k}}-\mathrm{i}\dfrac{\partial u}{\partial x_{4k+1}}-\mathrm{j}\dfrac{\partial u}{\partial x_{4k+2}}-\mathrm{k}\dfrac{\partial u}{\partial x_{4k+3}})\\
&=(\bigtriangledown_{(2l)0}+\mathrm{j}\bigtriangledown_{(2l+1)0}).(\bigtriangledown_{(2k+1)1}-\mathrm{j}\bigtriangledown_{(2k+1)0})\\
&= 2(\Delta_{(2l)(2k+1)}u+\mathrm{j}\Delta_{(2l+1)(2k+1)}u).
\end{array}$$
 For any $A\in Gl_{\mathbb{H}}(n)$  such that $A.\tilde{q}=q$ for $q,\tilde{q}$ in $\mathbb{H}^{n}$ and $\tilde{u}(\tilde{q})=u(A.\tilde{q})=u(q)$. By Corollary 3.1  in \cite{W3}
  we have
\begin{equation}\label{0000}
(\tilde{u}_{l\overline{k}}) =\Big(\dfrac{\partial^{2}\tilde{u}(\tilde{q})}{\partial\tilde{q}_{l}\partial\overline{\tilde{q}}_{k}}
\Big)=\overline{A}^{t}.\Big(\dfrac{\partial^{2}u(q)}{\partial q_{l}\partial \overline{q}_{k}}\Big).A= \overline{A}^{t}.(u_{\overline{l}k}).A
 \end{equation}
 We can choose  a suitable $A$ in $U_{\mathbb{H}}(n)$ such that the right hand side of (\ref{0000}) is a diagonal real  matrix. Hence
  $\tilde{\Delta}_{(2l)(2k+1)}\tilde{u}(\tilde{q})=0$ for $k \neq l$, $\tilde{\Delta}_{(2l+1)(2k+1)}\tilde{u}(\tilde{q})=0$ for all $k,l$ and $\tilde{\Delta}_{(2l)(2k)}\tilde{u}(\tilde{q})=\overline{\tilde{\Delta}_{(2l+1)(2k+1)}\tilde{u}(\tilde{q})}=0$. Therfore
   $\tilde{\Delta}\tilde{u}(\tilde{q})=\displaystyle{\sum _{l=0}^{n-1}\tilde{\Delta}_{(2l)(2l+1)}\tilde{u}(\tilde{q})\tilde{w}^{2l}\wedge \tilde{w}^{2l+1}}$
where $\lbrace \tilde{w}^{0},\cdots,\tilde{w}^{2n-1}\rbrace $ with  $\tilde{w}^{j}=A.w^{j}.$
Since  $(\tilde{u}_{l\overline{k}})$ and $(u_{l\overline{k}})$ are semblable matrices then they have same spectrum, namely
$\lbrace \lambda(u_{l\overline{k}})\rbrace=\lbrace 2\tilde{\Delta}_{(2l)(2l+1)}\tilde{u}(\tilde{q})\rbrace=\lbrace \lambda(\tilde{\Delta}\tilde{u})\rbrace$ for any choice of $A.$
Let $1\leq m\leq n$, from Corollary 2.2 and (2.10) in \cite{W2} follow
 $$\begin{array}{ll}
  ( \Delta u(q))^{m} \wedge \beta^{n-m} =(\Delta u(A.\tilde{q}))^{m}\wedge \beta^{n-m} &=(\tilde{\Delta}\tilde{u}(\tilde{q}))^{m} \wedge A.\beta^{n-m}\\
 &=(\tilde{\Delta}\tilde{u}(\tilde{q}))^{m} \wedge A.\Big(\sum_{l=0}^{n-1}\omega^{2l}\wedge \omega^{2l+1}\Big)^{n-m} \\
 &=(\tilde{\Delta}\tilde{u}(\tilde{q}))^{m} \wedge \Big(\sum_{l=0}^{n-1}A\omega^{2l}\wedge A\omega^{2l+1}\Big)^{n-m} \\
 &=(\tilde{\Delta}\tilde{u}(\tilde{q}))^{m} \wedge \Big(\sum_{l=0}^{n-1}\tilde{\omega}^{2l}\wedge  \tilde{\omega}^{2l+1})^{n-m}  \\
&= m!(n-m)!\tilde{S}_{m}(\tilde{\Delta}_{(2l)(2l+1)}\tilde{u} (\tilde{q}))\Omega_{2n}  \\
&=\dfrac{m!(n-m)!}{2^{m}}\tilde{S}_{m}((u_{l\overline{k}}))\Omega_{2n}.
\end{array}$$
So the result is proven.
\end{proof}
\begin{pro}\label{11}
For $\alpha_{1},\ldots,\alpha_{p}\in \widehat{\Gamma}_{m}\;,p\leq m$ we have $\alpha_{1}\wedge \ldots \wedge\alpha_{p}\wedge \beta^{n-m}\geq 0.$
\end{pro}
 \begin{proof}
 As in the proof of Proposition 2.1 in \cite{B}.\\
 We need to show that $\eta_{1}^{\ast}\tilde{\omega}^{0}\wedge \eta_{1}^{\ast}\tilde{\omega}^{1}\wedge  \ldots \wedge\eta_{m-p}^{\ast}\tilde{\omega}^{0}\wedge \eta_{m-p}^{\ast}\tilde{\omega}^{1}\wedge\alpha_{1}\wedge \ldots \wedge\alpha_{p}\wedge \beta^{n-m}\geq 0$
 for any $\eta_{j}^{\ast}\tilde{\omega}^{0}\wedge \eta_{j}^{\ast}\tilde{\omega}^{1}$ with $j=1,\ldots,m-p$ where $\lbrace \tilde{\omega}^{0},\tilde{\omega}^{1}\rbrace$ is a basis of $\mathbb{C}^{2}$ and $\eta_{j}^{\ast}:\mathbb{C}^{2}\rightarrow  \mathbb{C}^{2n} $ is the induced $\mathbb{C}$-linear pulling back transformation
of $ \eta_{j}.\omega,$  (cf. pages 12, 13 in \cite{W2} for more details.)
But since $\eta_{j}^{\ast}\tilde{\omega}^{0}\wedge \eta_{j}^{\ast}\tilde{\omega}^{1}\in
\widehat{\Gamma}_{n}\subset \widehat{\Gamma}_{m}$ (this is because $(\eta_{j}^{\ast}\tilde{\omega}^{0}\wedge \eta_{j}^{\ast}\tilde{\omega}^{1})^{2}=0$ ), we
may assume that $p=m$. Then the proposition follows from the Garding
inequality (\ref{8}).
 \end{proof}
  \begin{defi}
  A $(2n-2k)$-current $T$ $(k\leq m)$ is called $m$-positive if for $\alpha_{1},\ldots,\alpha_{k}\in \widehat{\Gamma}_{m}$, we have
\begin{equation}
\alpha_{1}\wedge \ldots \wedge\alpha_{k}\wedge T \geq 0. \label{2}
\end{equation}
 \end{defi}
  \begin{defi}
 A  function $u$  is called $m$-subharmonic  if it is subharmonic and, for any
 $\alpha_{1},\cdots,\alpha_{m-1}$ in $\widehat{\Gamma}_{m}$ the current $ \Delta u\wedge\alpha_{1}\wedge\cdots\wedge \alpha_{m-1}\wedge \beta ^{n-m}$ is $m$-positive.
 \end{defi}
The following basic properties of $m$-subharmonic functions follow
immediately from Proposition \ref{11}.
 \begin{pro}\cite[Proposition 4.3]{LW}\label{pro4} 
 \begin{enumerate}
 \item[1)] If $u,v \in \mathcal{SH}_{m}(\Omega)$, then $u+v  \in \mathcal{SH}_{m}(\Omega).$
 \item[2)] If $u  \in \mathcal{SH}_{m}(\Omega)$ and $ \gamma:\mathbb{R}\longrightarrow \mathbb{R}$ is a convex increasing function, then $\gamma \circ u  \in \mathcal{SH}_{m}(\Omega).$
 \item[3)] If $u  \in \mathcal{SH}_{m}(\Omega)$, then the standard regularizations $u*\rho_{\epsilon} \in \mathcal{SH}_{m}(\Omega_{\epsilon})$, where $\Omega_{\epsilon}:=\lbrace z/dis(z,\partial\Omega)> \epsilon\rbrace \;\;0< \epsilon \ll 1$.
 \item[4)] If $\lbrace u_{j}\rbrace \subset  \mathcal{SH}_{m}(\Omega)$ is locally uniformaly bounded from above, then $(\sup_{j}u_{j})^{*}  \in \mathcal{SH}_{m}(\Omega)$, where $v^{*}$ denotes the regularization of $v$.
 \item[5)] $\mathcal{PSH}(\Omega)= \mathcal{SH}_{n}(\Omega)\subset \cdots \subset \mathcal{SH}_{1}(\Omega)= \mathcal{SH}(\Omega).$
 \item[6)] Let $ \omega$ be a non-empty proper open subset of $ \Omega,\; u  \in \mathcal{SH}_{m}(\Omega),\;v  \in \mathcal{SH}_{m}(\omega)$ and $\displaystyle{\limsup_{q\longrightarrow \varsigma}}v(q)\leq u(\varsigma)$ for each $\varsigma \in  \partial\omega \cap \Omega$, then
 \begin{equation}
  W:=  \left \{
 \begin{array}{ll}
 \max\lbrace u,v\rbrace,\;\;\;in\;\; \omega \\
 u, \;\;\;\;\;\;\;\;\;in\;\; \Omega\setminus\omega \\
\end{array}
  \in \mathcal{SH}_{m}(\Omega).\right.
   \end{equation}
\end{enumerate}
 \end{pro}
Now for a continuous $m$-subharmonic functions, we can inductively define a
closed nonnegative current
\begin{equation}\label{15}
\Delta u_{1}\wedge \ldots \wedge \Delta u_{p}\wedge \beta^{n-m}:=\Delta(u_{1} \Delta u_{2} \wedge \ldots \wedge \Delta u_{p} \wedge \beta^{n-m})
\end{equation}
$$where \;\;\;\;u_{1},\ldots ,u_{p} \in \mathcal{SH}_{m}(\Omega)\cap C(\bar{\Omega}) \;\;and\;\;p \leq m .$$
For any test $(2m-2p)$-form $\psi$ on $ \Omega$, (\ref{15}) can be rewritten as
\begin{equation}\label{16}
\displaystyle{\int_{\Omega}\Delta u_{1}\wedge \ldots\wedge\Delta u_{p}\wedge\psi \wedge\beta^{n-m}=\int_{\Omega}u_{1}\Delta u_{2}\wedge \ldots\wedge\Delta u_{p}\wedge\Delta\psi \wedge\beta^{n-m}}.
\end{equation}
We can also define a nonnegative current
 \begin{equation}\label{17}
 d_{0}( u_{0}-u_{1}) \wedge d_{1}(u_{0}-u_{1}) \wedge \Delta u_{2} \wedge  \ldots
\wedge \Delta  u_{p} \wedge \beta^{n-m}
\end{equation}
$$where \;\;\;u_{0},u_{1},\ldots ,u_{p} \in \mathcal{SH}_{m}(\Omega)\cap C(\bar{\Omega})\;\;and\;\;p \leq m .$$
\begin{pro}\label{pro11}
The operators (\ref{15}) and (\ref{17}) are continuous for
locally uniformly convergent sequences in $ \mathcal{SH}_{m}(\Omega)\cap C(\bar{\Omega}).$
\end{pro}
\begin{proof}
 It is enough to prove the continuity of the operator
 $$(\mathcal{SH}_{m}\cap C)^{p}\ni (u_{1},\ldots ,u_{p})\mapsto u_{1}\Delta u_{2}\wedge \ldots \wedge \Delta u_{p}\wedge \beta^{n-m}.$$
  This follows inductively from the fact that the coefficients of a positive
current are Radon measures and since the convergence is uniform.
\end{proof}
\begin{lem}[Chern–Levine–Nirenberg-type estimate]\label{lem2}
 Let $\Omega$ be a domain in $\mathbb{H}^{n}$.
Let $K,L$  be compact subsets of $\Omega$ such that $L$ is contained in the interior of $K$, then there exists a constant
$C$ depending only on $K,L$ such that for any $u_{1},\ldots,u_{k} \in \mathcal{SH}_{m}(\Omega)\cap C^{2}(\Omega)\;\;(k \leq m)$ and for any  $m$ positive closed $(2n-2p)$-current $T$ with $(k \leq p)$, one has
$$\Vert \Delta u_{1}\wedge \ldots \wedge \Delta u_{k}\wedge T \Vert_{L} \leq C\prod_{i=1}^{k}\Vert u_{i} \Vert_{L^{\infty}(K)}\Vert T \Vert_{K},$$
where $\Vert T \Vert_{K}:=\int_{K}T\wedge \beta_{n}^{p}$.
Note that this estimate also holds for any $u_{1},\ldots,u_{p} \in \mathcal{SH}_{m}(\Omega)\cap L_{loc}^{\infty}(\Omega).$
\end{lem}
\begin{proof}
See Corollary 5.1 and Theorem 5.2 in \cite{LW}.
\end{proof}
\subsection{Quasicontinuity of $m$-subharmonic functions.}
Let $\Omega$ be an open set of $\mathbb{H}^{n}$, and let $K$ be a compact subset of $\Omega$, the quaternionic  capacity of $K$ in $\Omega$ is defined by:
\begin{equation}\label{18}
C_{m}(K)=C_{m}(K,\Omega):= \sup\Big\{\int_{K}(\Delta u)^{m}\wedge \beta^{n-m}\;\;:u\in \mathcal{SH}_{m}(\Omega),-1\leq u\leq 0\Big \},
\end{equation}
 and $$C_{m}(E)=C_{m}(E,\Omega):=\sup\Big\{ C_{m}(K):\;\;K \;\;\mbox{is\; a\; compact\; subset\; of \; E}  \Big \}.$$
As in the complex case, the following proposition can be proved.
\begin{pro}\label{pro12}
\begin{enumerate}
\item[1)] If $E_{1}\subseteq E_{2}$, then $C_{m}(E_{1})\subseteq C_{m}(E_{2}).$
\item[2)] If $E\subseteq \Omega_{1}\subseteq \Omega_{2}$, then $C_{m}(E,\Omega_{1})\geq C_{m}(E,\Omega_{2}).$
\item[3)] $C_{m}(\cup_{j=1}^{\infty}E_{j})\leq \sum_{j=1}^{\infty}C_{m}(E_{j}).$
\item[4)] If $E_{1}\subseteq E_{2}\subseteq \cdots \; are\; borel\; sets\; in\; \Omega\;,\; then\; C_{m}(\cup_{j}E_{j})=\displaystyle{\lim _{j\rightarrow \infty}}C_{m}(E_{j}).$
\end{enumerate}
\end{pro}
Before we prove the quasicontinuity, we need some preparations first. We are going to generalize the
convergence of the currents $ \Delta v_{1}\wedge \ldots \wedge \Delta v_{k}\wedge \beta^{n-k}$  slightly to $v_{0}\Delta v_{1}\wedge \ldots \wedge \Delta v_{k}\wedge \beta^{n-k}$
 for $v_{0},\ldots ,v_{k}\in \mathcal{SH}_{m}(\Omega)\cap L_{loc}^{\infty}(\Omega)$.
 In order to prove its continuous on decreasing sequences, we need the
following lemma.
\begin{lem}\label{lemma}
 Let $T$ be a  $m$-positive $(2n-2k)$-current on $\Omega$. Suppose that $T=0$ outside a compact subset $K$ of $\Omega$ and
$T(\Delta \psi)=0$ for all $\psi \in  D^{2k-2}(\Omega)$, then $T=0$.
\end{lem}
\begin{proof}
The same proof of Lemma 3.1 in \cite{W1}.
\end{proof}
\begin{pro}\label{pro3}
Let $v^{0},\ldots,v^{k}\in \mathcal{SH}_{m}(\Omega)\cap L_{loc}^{\infty}(\Omega)$, let $(v^{0})_{j},\ldots,(v^{k})_{j}$ be a decreasing sequences of $m$-subharmonic functions in $\Omega$ such that $\displaystyle{\lim _{j\rightarrow \infty}}v_{j}^{t}=v^{t}$ pointwisely in $\Omega$ for $t=0,\ldots,k$. Then for $k\leq m \leq n$ the currents $v_{j}^{0}\Delta v_{j}^{1}\wedge \ldots \wedge \Delta v_{j}^{k}\wedge \beta^{n-m}$ converge weakly to $v^{0}\Delta v^{1}\wedge \ldots \wedge \Delta v^{k}\wedge \beta^{n-m}$ as $j$ tends to $\infty$.
\end{pro}
\begin{proof}
By \cite[Theorem 5.2]{LW} and Lemma \ref{lemma}, we can repeat the same proof as in \cite[Proposition 3.2]{W1}.
\end{proof}
Let $u,v\in C^{2}(\Omega)$, define $$\gamma(u,v):=\dfrac{1}{2}(d_{0}u\wedge d_{1}v-d_{1}u\wedge d_{0}v):=\dfrac{1}{2}\sum_{i,j=0}^{2n-1}(\bigtriangledown_{i0}u\bigtriangledown_{j1}v-\bigtriangledown_{i1}u\bigtriangledown_{j0}v)\omega^{i}\wedge \omega^{j}.$$
In particular, $\displaystyle{\gamma(u,u)=d_{0}u\wedge d_{1}u=\dfrac{1}{2}\sum_{i,j=0}^{2n-1}(\bigtriangledown_{i0}u\bigtriangledown_{j1}u-\bigtriangledown_{i1}u\bigtriangledown_{j0}u)\omega^{i}\wedge \omega^{j}}.$
\begin{pro}{\cite[Proposition 3.3]{W1}}
Let $u\in C^{2}(\Omega)$, then $d_{0}u\wedge d_{1}u$ is a positive $2$-form.
\end{pro}
\begin{cor}\label{cor1}
Let $u,v \in C^{2}(\Omega)$ and let $\chi$ be a strongly positive $(2n-2)$-form, then
$$\Big \vert \int_{\Omega}\gamma(u,v)\wedge \chi \Big \vert^{2}\leq \int_{\Omega}\gamma(u,u)\wedge \chi.\int_{\Omega}\gamma(v,v)\wedge \chi.$$
\end{cor}
\begin{lem}\label{lem3}
Let $0<r_{1}<r_{2},\;a\in \Omega,\;B_{1}=B(a,r_{1}) \;and \;B_{2}=B(a,r_{2})$. Suppose that $\lbrace u_{j}\rbrace_{j\in \mathbb{N}}\subseteq \mathcal{SH}_{m}(\Omega)\cap C^{2}(B_{2})$ converge decreasingly to $u \in \mathcal{SH}_{m}(\Omega)\cap L_{loc}^{\infty}(B_{2})$, then
\begin{equation*}
\displaystyle{
\lim _{j\longrightarrow\infty}\sup \Big\lbrace\int_{\overline{B}_{1}}(u_{j}-u)(\Delta v)^{m}\wedge \beta^{n-m}\;\;,v \in \mathcal{SH}_{m}(B_{2})\;,-1\leq v\leq 0\Big \rbrace =0.}
\end{equation*}
\end{lem}
\begin{proof}
Without loss of generality, we may suppose that $-1\leq u_{j},u <0,\;u_{j}=u_{k}\;on\;B_{2}\setminus B_{1},\;for\;all\;j,k.$ As in the proof of Propostion 3.2 in \cite{W1},
let $v \in \mathcal{SH}_{m}(B_{2}),\;-1 \leq v < 0 $ and $s_{1},s_{2}$ be such $ 0 <r_{1}< s_{1} < s_{2}< r_{2}$. Define $\rho(q)=(\parallel q-a\parallel /s_{1})^{\alpha}-1,\;q \in \mathbb{H}^{n}$ where $\alpha$ is taken such $\rho\leq\dfrac{-1}{2}$ on $\overline{B}_{1}$. Define
\begin{equation}
w=\left\{
     \begin{array}{ll}
       \max \{ \rho,\dfrac{v}{2}\} \;\;;\;\;\;\;\;B_{2} \\
        \rho\;\;\;;\;\;\;\;\;\;\;\;\;out\;\;\; of \;B_{2}\\
     \end{array}
   \right.\;\;\;\;\;\mbox{and}\; \;\tilde{w}=\max \{ \rho,0 \}.
\end{equation}
Then $w,\tilde{w}\in \mathcal{SH}_{m}(\mathbb{H}^{n})$ and $w=\dfrac{v}{2}$ in $\bar{B}_{1}$, $w=\tilde{w}=\rho$ in $\mathbb{H}^{n} \setminus B(a,s_{1})$.
If $0<\epsilon < min\lbrace s_{2}-s_{1},s_{1}-r_{1}\rbrace$, then $w_{\epsilon}=w*\chi_{\epsilon}=\tilde{w}*\chi_{\epsilon} =\tilde{w}_{\epsilon}$ in $\mathbb{H}^{n}\setminus B(a,s_{2})$ and $\tilde{w}_{\epsilon}=1$ in $B_{1}$.
Note that $$\Delta=\dfrac{1}{2}(d_{0}d_{1}-d_{1}d_{0})\;\;\mbox{and}\;\; d_{0}d_{1}=-d_{1}d_{0},$$
then by (\ref{15}) we have
$$\begin{array}{ll}
\displaystyle{\int}_{\overline{B}_{1}}(u_{j}-u_{k})(\Delta \omega_{\varepsilon})^{m}\wedge \beta^{n-m}&=\displaystyle{\dfrac{1}{2}\int}_{\overline{B}_{1}}(u_{j}-u_{k})(d_{0}d_{1} \omega_{\varepsilon}-d_{1}d_{0}\omega_{\varepsilon}) \wedge(\Delta \omega_{\varepsilon})^{m-1} \wedge \beta^{n-m}\\ &=  \displaystyle{\dfrac{1}{2}\int}_{\overline{B}_{1}}(u_{j}-u_{k})d_{0}[d_{1}\omega_{\varepsilon} \wedge(\Delta \omega_{\varepsilon})^{m-1}\wedge \beta^{n-m}]\\&\;\;\;\;\;\;-\displaystyle{ \dfrac{1}{2}\int}_{\overline{B}_{1}}(u_{j}-u_{k})d_{1}[d_{0}\omega_{\varepsilon}\wedge(\Delta\omega_{\varepsilon})^{m-1}\wedge \beta^{n-m}]\\
& = -\displaystyle{\dfrac{1}{2}\int}_{\overline{B}_{1}}d_{0}(u_{j}-u_{k}) \wedge d_{1}\omega_{\varepsilon}\wedge(\Delta \omega_{\varepsilon})^{m-1}\wedge \beta^{n-m}\\ &\;\;\;\;\;\;+\displaystyle{\dfrac{1}{2} \int}_{\overline{B}_{1}}d_{1}(u_{j}-u_{k})\wedge d_{0}\omega_{\varepsilon}\wedge(\Delta\omega_{\varepsilon})^{m-1}\wedge\beta^{n-m}\\
&=-\displaystyle{\int}_{\overline{B}_{1}}\gamma(u_{j}-u_{k},\omega_{\varepsilon}) \wedge(\Delta\omega_{\varepsilon})^{m-1}\wedge \beta^{n-m}\\
&=-\displaystyle{\int}_{\overline{B}_{1}}\gamma(u_{j}-u_{k},\omega_{\varepsilon}-\tilde{\omega}_{\varepsilon})\wedge(\Delta\omega_{\varepsilon})^{m-1}\wedge\beta^{n-m}.\\
\end{array}$$
The third identity above follows from Stokes-type formula Lemma  \ref{lem1} and the fact that $u_{j} = u_{k}$ on the
boundary. The fourth and fifth identities follow from formula of $\gamma(u,v)$ and the fact that $\omega_{\varepsilon} = 1$ in a neighborhood
of $B_{1}$, respectively.
Now suppose that $k\geq j$, by Corollary \ref{cor1}  and by repeating the process, we have
$$\begin{array}{ll}
&\Big \vert \displaystyle{\int}_{\overline{B}_{1}}\gamma(u_{j}-u_{k},\omega_{\varepsilon}-\tilde{\omega}_{\varepsilon})\wedge(\Delta\omega_{\varepsilon})^{m-1}\wedge\beta^{n-m} \Big   \vert^{2}\\&
 \leq \displaystyle{\int}_{\overline{B}_{1}}\gamma(u_{j}-u_{k},u_{j}-u_{k})\wedge(\Delta\omega_{\varepsilon})^{m-1}\wedge\beta^{n-m}.\displaystyle{\int}_{\overline{B}_{1}}\gamma(\omega_{\varepsilon}-\tilde{\omega}_{\varepsilon},\omega_{\varepsilon}-\tilde{\omega}_{\varepsilon})\wedge(\Delta\omega_{\varepsilon})^{m-1}\wedge\beta^{n-m}\\
& \leq\displaystyle{\int}_{\overline{B}_{1}}(u_{j}-u_{k})\Delta(u_{j}-u_{k})\wedge(\Delta \omega_{\varepsilon})^{m-1}\wedge \beta^{n-m}.\displaystyle{\int}_{\overline{B}(a,s_{2})}(\omega_{\varepsilon}-\tilde{\omega}_{\varepsilon} )\Delta(\omega_{\varepsilon}-\tilde{\omega}_{\varepsilon})\wedge(\Delta \omega_{\varepsilon})^{m-1}\wedge \beta^{n-m}\\
 &\leq C^{'}\displaystyle{\int}_{\overline{B}_{1}}(u_{j}+u_{k})\Delta(u_{j}-u_{k})\wedge(\Delta \omega_{\varepsilon})^{m-1}\wedge \beta^{n-m}
 \end{array}$$
 where the last inequality follows from Lemma \ref{lem2}. Note that $\Delta(u_{j}-u_{k})\wedge(\Delta \omega_{\varepsilon})^{m-1}\wedge \beta^{n-m}$ need not be positive,
it suffices to estimate the integral obtained by replacing $\Delta(u_{j}-u_{k}) $ by $\Delta(u_{j}+u_{k})$. This term can then be estimated by Corollary \ref{cor1}  and Lemma \ref{lem2} as before.\\
Thus the power of $\Delta \omega_{\varepsilon}$ has been reduced to $n-1$. Repeating this argument $m-1$ times to get
$$\displaystyle{  \int}_{\overline{B}_{1}}(u_{j}-u_{k})(\Delta \omega_{\varepsilon})^{m}\wedge \beta^{n-m} \leq C\displaystyle{\Big [ \int}_{\overline{B}_{1}}(u_{j}-u_{k})(\Delta(u_{j}+u_{k}))^{m} \wedge \beta^{n-m} \Big ]^{D}$$
where $C,D >0 $ are constants independent of $u_{j},u_{k},\omega_{\varepsilon}$.
Let $\eta \in C_{0}^{\infty}(B_{2})$ be nonnegative and such that $\eta \equiv 1$ in a neighborhood of $B_{1}$. By Proposition \ref{pro3}, we
have
$$\begin{array}{ll}
\displaystyle{\int}_{\overline{B}_{1}}(u_{j}-u_{k})(\Delta \omega_{ \varepsilon})^{m}\wedge \beta^{n-m}
&= \displaystyle \int_{B_{2}} \eta(u_{j}-u_{k})( \Delta \omega_{ \varepsilon})^{m} \wedge \beta^{n-m} \displaystyle\underset{\varepsilon \rightarrow 0}{ \longrightarrow}  \displaystyle  \int_{B_{2}} \eta(u_{j}-u_{k})(\Delta \omega)^{m} \wedge \beta^{n-m}\\
&= \dfrac{1}{2^{m}}\displaystyle  \int_{\overline{B}_{1}}\eta(u_{j}-u_{k})(\Delta v)^{m}\wedge \beta^{n-m}.
\end{array}$$
Therefore
$$\displaystyle{\int}_{\overline{B}_{1}}\eta(u_{j}-u_{k})(\Delta v)^{m}\wedge \beta^{n-m}\leq 2^{m}C \Big[\displaystyle{ \int}_{B_{1}}\eta(u_{j}-u_{k})(\Delta u_{j}+u_{k})^{m}\wedge \beta^{n-m} \Big]^{D}.$$
let $k\longrightarrow\infty$, by applying Proposition \ref{pro3} and taking the supremum, we have
$$\sup \Big \lbrace\int_{\overline{B}_{1}}(u_{j}-u)(\Delta v)^{m}\wedge \beta^{n-m},\;v \in \mathcal{SH}_{m}(B_{2})\;,-1\leq v\leq 0 \Big \rbrace  \leq 2^{m}C \Big[\int_{B_{1}}\eta(u_{j}-u_{k})(\Delta u_{j}+u)^{m}\wedge \beta^{n-m}\Big]^{D}$$
Then let $j\longrightarrow\infty $ and apply Proposition \ref{pro3} to get the conclusion.
\end{proof}
\begin{thm} \label{th1}
For a $m$-subharmonic function $u$ defined in $\Omega$ and a positive number $\varepsilon$
one can find an open set $U \subset \Omega$ with $C_{m}(U,\Omega)<\varepsilon$ and such that $u$ restricted to
$\Omega\setminus U$ is continuous.
\end{thm}
\begin{proof}
By Proposition \ref{pro12}, we can suppose that $\Omega$ is a ball $B(a,\eta)$ for some $a\in \mathbb{H}^{n}$, $\eta> 0$. Let ${\eta_{j}}$ be an increasing sequence of positive numbers such that $\displaystyle{\lim_{j\rightarrow \infty}\eta_{j}}=\eta$ and let $B_{j}=B(a,\eta_{j})$. Now fix $\varepsilon$, by Lemma \ref{lem3}, we can choose a decreasing sequence ${u_{j}}\subseteq \mathcal{SH}_{m}(B_{j})\cap C^{\infty}(\overline{B}_{j})$ that converges to $u$ such that:
$$sup \Big\lbrace\int_{\overline{B}_{j}}(u_{j}-u)(\Delta u)^{m}\wedge \beta^{n-m}\;\;:u\in \mathcal{SH}_{m}(\Omega),-1\leq u\leq 0 \Big\rbrace <2^{-j}.$$
Let $$\omega_{j}=\lbrace q\in B_{j}\;:\;(u_{j}-u)(q)>j^{-1}\rbrace\;\; and\;\; \omega=\cup_{j\geq k}\omega_{j},$$
then $\omega$ is open  and we have $C_{m}(\omega_{j})<j2^{-j}$.\\
Choose $k$ such that $\displaystyle{\sum_{j=k}^{\infty}}(j2^{-j})< \varepsilon,$ from Proposition \ref{pro12} follows $C_{m}(\omega)\leq \displaystyle{\sum_{j=k}^{\infty}}C_{m}(\omega_{j})\leq \displaystyle{\sum_{j=k}^{\infty}}(j2^{-j})< \varepsilon$.
Since $u_{j}$ converges to $u$ locally uniformly in $\Omega\setminus\omega$ , then $u$ is continuous on $\Omega\setminus\omega$.
\end{proof}
\begin{lem}\label{lem4}
Let $\lbrace u_{j}\rbrace_{j\in \mathbb{N}}$ be a locally uniformly bounded sequence in $\mathcal{SH}_{m}\cap L_{loc}^{\infty}(\Omega)$ that increases to $u \in \mathcal{SH}_{m}\cap L_{loc}^{\infty}
(\Omega)$ almost every where in $\Omega$ and let $v^{1},\ldots,v^{m} \in \mathcal{SH}_{m}(\Omega)\cap L_{loc}^{\infty}(\Omega)$. Then the currents $u_{j}\Delta v^{1}\wedge \ldots \wedge \Delta v^{m}\wedge \beta^{n-m}$ converge to $u\Delta v^{1}\wedge \ldots \wedge \Delta v^{m}\wedge \beta^{n-m}$ as $j\rightarrow\infty.$
\end{lem}
\begin{proof}
 As the problem is local, we can assume that all functions have values in $(-1, 0)$. Let $ v=\lim u_{j},$ by
Lebesgue monotone convergence theorem,  we have $v\Delta v^{1}\wedge \ldots \wedge \Delta v^{m}\wedge \beta^{n-m}$ is the weak limit of $u_{j}\Delta v^{1}\wedge \ldots \wedge \Delta v^{m}\wedge \beta^{n-m}.$    Since $u\Delta v^{1}\wedge \ldots \wedge \Delta v^{m}\wedge \beta^{n-m}\geq  v\Delta v^{1}\wedge \ldots \wedge \Delta v^{m}\wedge \beta^{n-m},$
 it is enough to prove that the two measures have locally the same
mass.
Denote $v^{i}\ast\chi_{\varepsilon}$ by $ v_{\varepsilon}^{i}$, then $v_{\varepsilon}^{i} \searrow v^{i}$ on B as $ \varepsilon \rightarrow 0.$ As in the proof of Proposition \ref{pro3}, we can assume that $v_{\varepsilon}^{i}=v^{i}$
 in $B \setminus K$, where $K$ is a compact subset of $B$. Let $\eta \in C_{0}^{\infty}(B)\; \eta =1$ near $K$, by the quasicontinuity
of $v^{m}$ there exists an open subset $\omega$ of $\Omega$ such that $C(\omega)<\varepsilon$ and $v^{m}\vert_{\Omega \setminus \omega}$ is continuous. Then
$$\begin{array}{ll}
\displaystyle{  \int}_{\Omega}\eta u_{j}\Delta v^{1}\wedge \ldots \wedge\Delta v^{m-1}\wedge \Delta(v^{m}-v_{\varepsilon}^{m})\wedge \beta^{n-m}
&=\displaystyle{  \int}_{\Omega}u_{j}\Delta v^{1}\wedge \ldots \wedge\Delta v^{m-1}\wedge \Delta(v^{m}-v_{\varepsilon}^{m})\wedge \beta^{n-m}\\
&= \displaystyle{  \int}_{\Omega}(v^{m}-v_{\varepsilon}^{m}) \Delta v^{1}\wedge \ldots \wedge\Delta v^{m-1} \wedge \Delta u_{j}\wedge \beta^{n-m}\\
&=\displaystyle \int_{\Omega \setminus \omega}(v^{m}-v_{\varepsilon}^{m}) \Delta v^{1}\wedge \ldots \wedge\Delta v^{m-1}\wedge \Delta u_{j}\wedge \beta^{n-m} \\
&\;\;\;\;\;+ \displaystyle \int_{\omega}(v^{m}-v_{\varepsilon}^{m}) \Delta v^{1}\wedge \ldots \wedge\Delta v^{m-1}\wedge \Delta u_{j}\wedge \beta^{n-m}.\\
\end{array}$$
Since all functions have values in $(-1,0)$, by definition (\ref{17}) we have
$$
 \Big \vert \displaystyle{  \int}_{\Omega}(v^{m}-v_{\varepsilon}^{m}) \Delta v^{1}\wedge \ldots \wedge \Delta v^{m-1} \wedge \Delta u_{j}\wedge \beta^{n-m} \Big \vert
\leq \Vert v_{\varepsilon}^{m} \Vert_{L^{\infty}(\omega)}C(\omega)< C^{'}\varepsilon.
$$
Therefore $\displaystyle{  \int}_{\Omega}\eta u_{j}\Delta v^{1}\wedge \ldots \wedge\Delta v^{m-1}\wedge \Delta(v^{m}-v_{\varepsilon}^{m})\wedge \beta^{n-m} $ converges to $0$ uniformly with respect to $j$ when $\varepsilon \rightarrow 0$, by
Lemma \ref{lem2} and  repeating this process $(n-1)$ times we get
$$\begin{array}{ll}
\displaystyle{  \int}_{\Omega}\eta u_{j}\Delta v^{1}\wedge \ldots \wedge\Delta v^{m}\wedge  \beta^{n-m}
&\geq  \displaystyle{  \int}_{\Omega} \eta u_{j}\Delta v^{1}\wedge \ldots \wedge\Delta v^{m-1}\wedge \Delta v_{\varepsilon_{m}}^{m}\wedge \beta^{n-m}-\delta\\
&\geq  \displaystyle{  \int}_{\Omega} \Delta v_{\varepsilon_{1}}^{1}\wedge \ldots \wedge\Delta v_{\varepsilon_{m}}^{m} \wedge \beta^{n-m}-n\delta.
\end{array}$$
Then
$\displaystyle{\liminf _{j\longrightarrow \infty} \int_{\Omega}\eta u_{j}\Delta v^{1}\wedge \ldots \wedge\Delta v^{m}\wedge  \beta^{n-m}}\geq \displaystyle{  \int}_{\Omega}u \Delta v_{\varepsilon_{1}}^{1}\wedge \ldots \wedge\Delta v_{\varepsilon_{m}}^{m} \wedge \beta^{n-m}-n\delta .$\\
Now let $\delta,\varepsilon_{1},\ldots,\varepsilon_{m} \longrightarrow 0$ and use Proposition \ref{pro3} for $v_{j}^{0}=v^{0}=1$ to get the desired conclusion.
\end{proof}
\begin{pro}\label{pro5}
Let $\lbrace u_{j}\rbrace_{j \in \mathbb{N}}$ be a sequence in $\mathcal{SH}_{m}\cap L_{loc}^{\infty}
(\Omega)$ that increases to $u \in \mathcal{SH}_{m} \cap L_{loc}^{\infty}$ almost everywhere in $\Omega$ (with respect to lebesgue measure). Then the currents $(\Delta u_{j})^{m}\wedge \beta^{n-m}$ converge weakly to $(\Delta u)^{m}\wedge \beta^{n-m}$ as $j\rightarrow \infty.$
\end{pro}
\begin{proof}
We need to prove that for any $\varphi \in C_{0}^{\infty}(\Omega)$,
$$\;\; \displaystyle{\lim_{j\longrightarrow \infty}\int_{\Omega}\varphi(\Delta u_{j})^{m}\wedge \beta^{n-m}=\int_{\Omega}\varphi(\Delta u)^{m}\wedge \beta^{n-m}}.$$
Take $ \eta\in C_{0}^{\infty}(\Omega)$ with $ \eta =1$ on the support of $\varphi$ and $ \varphi_{1},\varphi_{2}\in \mathcal{SH}_{m}\cap C^{\infty}$ with $\varphi=\varphi_{1}-\varphi_{2}$ (this can be
proved in the same way as the complex case). Then by (\ref{16}) we have
$$ \displaystyle{ \int_{\Omega}\varphi(\Delta u_{j})^{m} \wedge \beta^{n-m}}  = \displaystyle{ \int_{ \Omega}u_{j}(\Delta u)^{m-1}\wedge \Delta \varphi \wedge \beta^{n-m}}\\
 =  \displaystyle{ \int_{\Omega} \eta u_{j}(\Delta u)^{m-1} \wedge( \Delta \varphi_{1}- \Delta \varphi_{2})\wedge \beta^{n-m}}.$$
By monotonicity  for $k\leq j$, we get
\begin{equation}\label{19}
\begin{array}{ll}
\displaystyle{\int}_{\Omega}\eta u_{k}(\Delta u_{j})^{m-1}\wedge\Delta\varphi_{1}\wedge \beta^{n-m}
&\leq \displaystyle{\int}_{\Omega}\eta u_{j}(\Delta u_{j})^{m-1}\wedge\Delta\varphi_{1}\wedge \beta^{n-m}\\
&\leq \displaystyle{\int}_{\Omega}\eta u(\Delta u_{j})^{m-1}\wedge\Delta\varphi\wedge \beta^{n-m},\\
\end{array}
\end{equation}
\begin{equation}\label{20}
\hbox{if}\;\; (\Delta u_{j})^{m-1}\wedge \Delta \varphi_{1}\wedge\beta^{n-m}\longrightarrow(\Delta u)^{m-1}\wedge \Delta \varphi_{1}\wedge\beta^{n-m}\;\;as\;j\longrightarrow\infty.
  \end{equation}
Then from the quasicontinuity of $u_{k}$ follows
\begin{equation}\label{21}
\displaystyle{\lim_{j\longrightarrow \infty}\int_{\Omega}\eta u_{k}(\Delta u_{j})^{m-1}\wedge\Delta\varphi_{1}\wedge \beta^{n-m}=\int_{\Omega}\eta u(\Delta u_{j})^{m-1}\wedge\Delta\varphi_{1}\wedge \beta^{n-m}}.
 \end{equation}
By (\ref{19}) and (\ref{21}) we have
$$\begin{array}{ll}
 \displaystyle{\int_{\Omega}\eta u_{k}(\Delta u)^{m-1}\wedge\Delta\varphi_{1}\wedge \beta^{n-m}}
& \leq  \displaystyle{\liminf_{j\longrightarrow\infty}} \displaystyle{\int_{\Omega}\eta u_{j}(\Delta u_{j})^{m-1}\wedge\Delta\varphi_{1}\wedge \beta^{n-m}}      \\
& \leq  \displaystyle{\limsup_{j\longrightarrow\infty}} \displaystyle{ \int_{\Omega}\eta u(\Delta u_{j})^{m-1}\wedge\Delta\varphi_{1}\wedge \beta^{n-m}}      \\
& =  \displaystyle{\int_{\Omega}\eta u(\Delta u_{j})^{m-1}\wedge\Delta\varphi_{1}\wedge \beta^{n-m}}.
\end{array}$$
where the last identity follows from the quasicontinuity of $u$.\\ By Lemma \ref{lem3}, we get that
$\displaystyle{ \lim_{k\longrightarrow\infty}}\displaystyle{\int_{\Omega}\eta u_{k}(\Delta u)^{m-1}\wedge\Delta\varphi_{1}\wedge \beta^{n-m}=\int_{\Omega}\eta u(\Delta u)^{m-1}\wedge\Delta\varphi_{1}\wedge \beta^{n-m}}.$  Thus we have $$\displaystyle{ \lim_{j\longrightarrow\infty}\int_{\Omega}\eta u_{j}(\Delta u_{j})^{m-1}\wedge\Delta\varphi_{1}\wedge \beta^{n-m}=\int_{\Omega}\eta u(\Delta u)^{m-1}\wedge\Delta\varphi_{1}\wedge \beta^{n-m}}.$$
Repeating this process for $\displaystyle{\int}_{\Omega}\eta u_{j}(\Delta u_{j})^{m-1}\wedge\Delta\varphi_{2}\wedge \beta^{n-m} $ provided that (\ref{20}) holds for $ \varphi_{2}.$\\ If  $(\Delta u_{j})^{m-2}\wedge\Delta\varphi\wedge\Delta\psi\wedge\beta^{n-m} \rightarrow (\Delta u)^{m-2}\wedge\Delta\varphi\wedge\Delta\psi\wedge\beta^{n-m} $ weakly as $j\longrightarrow\infty$ for any $ \varphi,\psi \in C^{\infty}(\Omega)$ holds,  we could do this as above.
From Lemma \ref{lem4} and repeat this $n-1$ times follow  the proposition.
\end{proof}
\begin{thm}\label{th2}
Let $\Omega$ be a bounded domain in $\mathbb{H}^{n}$, let  $u,v \in \mathcal{SH}_{m}(\Omega)\cap L_{loc}^{\infty}$, if $\displaystyle{\liminf_{q\longrightarrow \xi}}(u(q)-v(q))\geq0 $ for $\xi \in \partial \Omega.$
  Then
    $\displaystyle{\int}_{\lbrace u< v\rbrace}(\Delta v)^{m}\wedge \beta^{n-m} \leq \displaystyle{\int}_{\lbrace u< v\rbrace}(\Delta u)^{m}\wedge \beta^{n-m}.$
\end{thm}
\begin{proof}
Assume first that  $u, v$ are continuous on $\Omega$.
 Since $u, v$ are continuous, we can suppose that $\Omega =\{u< v\}$ and $\displaystyle{ \liminf_{q\longrightarrow \xi}}(u(q)-v(q))=0$. By \cite[Theorem 5.3]{LW} we have $\displaystyle{\int_{ \Omega} (\Delta v)^{m}\wedge \beta^{n-m}} \leq \displaystyle{\int_{ \Omega} (\Delta u)^{m}\wedge \beta^{n-m}}.$
 By replacing $ u$ by $u+2\delta$ and taking $\delta\longrightarrow 0$, we can assume that for any $\zeta\in\partial\Omega\;,\;\displaystyle{\liminf_{\zeta\longleftarrow q \in \Omega}}(u(q)-v(q))\geq 2\delta > 0.$
So $E =\lbrace u < v+\delta \rbrace$ is a relatively compact subset in $\Omega$. Then we can find sequences $(u_{j})_{j},(v_{j})_{j}\subseteq\mathcal{SH}_{m}(\Omega)\cap C(\Omega_{1})$
such that $ u_{j} \searrow u, v_{j} \searrow v $ and $ u_{k} \geq v_{j}$ for all $j\geq k$ on $\Omega_{1}\setminus \ Int(E)$, where $\Omega_{1}$ is a neighborhood of $\overline{E}$. Take
$M> 2\max\lbrace \parallel u\parallel_{\Omega},\parallel v\parallel_{\Omega},\parallel u_{1}\parallel_{ \Omega_{1}},\parallel v_{1}\parallel_{ \Omega_{1}}\rbrace$ and apply the first part of the proof  to $u_{k}, v_{j}, \Omega_{1}$ to get
\begin{equation}\label{23}
\displaystyle{\int_{\lbrace u_{k}\leq v_{j}\rbrace} (\Delta v_{j})^{m}\wedge \beta^{n-m}}\leq\displaystyle{\int_{\lbrace u_{k}\leq v_{j}\rbrace} (\Delta u_{k})^{m}\wedge \beta^{n-m}}
 \end{equation}
for $j\geq k$.\\
 Take $\varepsilon> 0$, by Theorem \ref{th1} there exists an open set $G_{\varepsilon}\subseteq \Omega_{1}$ such that $C(G_{\varepsilon},\Omega_{1}) < \varepsilon$, and
$u,v$ are continuous on $\Omega_{1}\ G_{\varepsilon}$, where $C(G_{\varepsilon},\Omega_{1})$ is quaternionic capacity of $G_{\varepsilon}$ in $\Omega_{1}$ defined by (\ref{18}). By
Tietze's extension theorem, there exists $f\in C(\Omega_{1})$ such that $f\mid_{(\Omega_{1}\setminus G_{\varepsilon})} = v\mid_{(\Omega_{1} \setminus G_{\varepsilon})}$.\\
Note that $\lbrace u_{k}< f \rbrace\cup G_{\varepsilon}=\lbrace u_{k}<v\rbrace \cup G_{\varepsilon}$. By Fatou’s lemma and (\ref{23}) we have
\begin{equation}\label{24}
\begin{array}{ll}
\displaystyle{\int_{\lbrace u_{k}< v\rbrace} (\Delta v)^{m}\wedge \beta^{n-m}}&\leq  \displaystyle{\int_{\lbrace u_{k}< f\rbrace} (\Delta v)^{m}\wedge \beta^{n-m}}+\displaystyle{\int_{ G_{\varepsilon}} (\Delta v)^{m}\wedge \beta^{n-m}}\\
&\leq  \displaystyle{\liminf_{j\longrightarrow\infty}\int_{\lbrace u_{k}< f\rbrace} (\Delta v_{j})^{m}\wedge \beta^{n-m}} +\displaystyle{\int_{ G_{\varepsilon}} (\Delta v)^{m}\wedge \beta^{n-m}}\\
&\leq  \displaystyle{\liminf_{j\longrightarrow\infty}\bigg[\int_{\lbrace u_{k}< v_{j}\rbrace} (\Delta v_{j})^{m}\wedge \beta^{n-m}} +\displaystyle{\int_{ G_{\varepsilon}} (\Delta v_{j})^{m}\wedge \beta^{n-m}\bigg]}\\ &\;\;\;\;\;\;+\displaystyle{\int_{ G_{\varepsilon}} (\Delta v)^{m}\wedge \beta^{n-m}}\\
&\leq  \displaystyle{\lim_{j\longrightarrow\infty}\int_{\lbrace u_{k}<v_{j}\rbrace} (\Delta u_{k})^{m}\wedge \beta^{n-m}}+2M^{m}C(G_{\varepsilon},\Omega_{1})\\ &\leq  \displaystyle{\int_{\lbrace u_{k}<v\rbrace} (\Delta u_{k})^{m}\wedge \beta^{n-m}}+2M^{m}\varepsilon.
\end{array}
\end{equation}
On the other hand, we have
\begin{equation}\label{25}
\begin{array}{ll}
\displaystyle{\int_{\{u \leq v\}}(\Delta u)^{m}\wedge \beta^{n-m}} &\geq \displaystyle{\int_{\{u\leq v\}\cap G_{\varepsilon}}(\Delta u)^{m}\wedge \beta^{n-m}} \\
&\geq \displaystyle{\limsup_{k\longrightarrow \infty}}\displaystyle{\int_{\{u \leq v\}\cap G_{\varepsilon}}(\Delta u_{k})^{m}\wedge \beta^{n-m}}\\
& \geq \displaystyle{\limsup_{k\longrightarrow\infty}}\Big[\displaystyle{\int_{\{u \leq v\}}(\Delta u_{k})^{m}\wedge \beta^{n-m}}-\displaystyle{\int_{{G_{\varepsilon}}}(\Delta u_{k})^{m}\wedge \beta^{n-m}}\Big]\\
&\geq\displaystyle{ \limsup_{k\longrightarrow\infty}}\displaystyle{\int_{\{u_{k} \leq v\}}(\Delta u_{k})^{m}\wedge \beta^{n-m}}-M^{m}\varepsilon.
\end{array}
 \end{equation}
Let $k\longrightarrow\infty$ in (\ref{24}), then the left hand side of (\ref{24}) converges to $\displaystyle{\int}_{\{u < v\}}(\Delta v)^{m}\wedge \beta^{n-m}$. It follows from (\ref{24}) and
(\ref{25}) that
$\displaystyle{\int_{\{u < v\}}(\Delta v)^{m}\wedge \beta^{n-m}}\leq \displaystyle{\int_{\{u\leq v\}}(\Delta u)^{m}\wedge \beta^{n-m}}+3M^{m}\varepsilon.$
Since $\varepsilon$ is arbitrary, we get
\begin{equation}\label{26}
\displaystyle{\int_{\{u < v\}}(\Delta v)^{n}\wedge \beta^{n-m}}\leq \displaystyle{\int_{\{u \leq v\}}(\Delta v)^{n}\wedge \beta^{n-m}}.
 \end{equation}
Now replace $u$ by $u+\theta\;(\theta> 0)$ in (\ref{26}) to get $$
\displaystyle{\int_{\{u+\theta < v\}}(\Delta v)^{n}\wedge \beta^{n-m}}\leq\displaystyle{\int_{\{u+\theta \leq v\}}(\Delta( u+\theta))^{n}\wedge \beta^{n-m}}=\displaystyle{\int_{\{u+\theta \leq v\}}(\Delta u)^{n}\wedge \beta^{n-m}}.
$$
Note that $\{u+\theta < v\}\nearrow \{u< v\}$ as $\theta \searrow 0$, it follows the result.
\end{proof}
\begin{cor}\label{cor2}
Let $\Omega$ be a bounded domain in $\mathbb{H}^{n},$ let  $u,v \in \mathcal{SH}_{m}(\Omega)\cap L_{loc}^{\infty}$,
  if  $\displaystyle{ \liminf_{q\longrightarrow \xi}(u(q)-v(q))\geq0}$ for $\xi \in \partial\Omega$
  and $(\Delta u)^{m}\wedge \beta^{n-m} \leq (\Delta v)^{m}\wedge \beta^{n-m}$ in $ \Omega.$
 Then $u\geq v$ in $\Omega$.
\end{cor}
\begin{proof}
 Let $\rho(q) = \dfrac{\Vert q \Vert^{2}}{8}$, then $\Delta\rho(q) = \beta$. Define $v_{\varepsilon,\delta} = v + \varepsilon\rho -\delta$ with $\varepsilon,\delta > 0$ such that
$\varepsilon\rho -\delta < 0$, then $\displaystyle{ \liminf_{q\longrightarrow \xi}}(u(q)-v_{\varepsilon,\delta}(q))\geq0 $ for any $ \xi \in \partial \Omega$. If the set $\{u < v\}$ is nonempty, then
$\{u > v_{\varepsilon,\delta}\}$ is also nonempty for  $\varepsilon,\delta > 0$. The fact that any $ \mathcal{SH}_{m} $ function is subharmonic enables the
statement below: If $u_{1},u_{2} \in \mathcal{SH}_{m}(\Omega) $ and $ u_{1}=u_{2}$ almost everywhere in $\Omega$, then $u_{1}\equiv u_{2}$. So the Lebesgue
measure of the set $\{u < v_{\varepsilon,\delta}\}$ must be positive. Note that
$$\begin{array}{ll}
(\Delta v_{\varepsilon,\delta})^{m} &=(\Delta v)^{m} + (\Delta(\varepsilon\rho-\delta))^{m} +\sum_{i=1}^{m}\left(
    \begin{array}{c}
      m \\
      j \\
    \end{array}
  \right)(\Delta v)^{j}\wedge (\Delta(\varepsilon\rho -\delta))^{m-j}\\ &\geq(\Delta v)^{m} + (\Delta(\varepsilon\rho-\delta))^{m}\\ & =(\Delta v)^{m} + \varepsilon^{m}\beta^{m}.
\end{array}$$
Apply Theorem \ref{th2} to $v_{\varepsilon,\delta}$ and $u$, we  obtain
$$\begin{array}{ll}
\displaystyle{\int_{\{u < v_{\varepsilon,\delta}\}}(\Delta v)^{m}\wedge \beta^{n-m}}+ \displaystyle{\int_{\{u\leq v_{\varepsilon,\delta}\}}\varepsilon^{m}\beta^{n}}
& \leq \displaystyle{\int_{\{u < v_{\varepsilon,\delta}\}}(\Delta v_{\varepsilon,\delta})^{m}\wedge \beta^{n-m}}\\ & \leq \displaystyle{\int_{\{u < v_{\varepsilon,\delta}\}}(\Delta u)^{m}\wedge \beta^{n-m}}\\
&\leq \displaystyle{\int_{\{u < v_{\varepsilon,\delta}\}}(\Delta v)^{m}\wedge \beta^{n-m}},
\end{array}$$
which is a contradiction. Therefore the set $\{u < v\}$ is empty.
\end{proof}
\begin{cor}\label{cor3}
let  $u,v \in \mathcal{SH}_{m}(\Omega)\cap L_{loc}^{\infty}$, if $\displaystyle{\liminf_{q\longrightarrow \xi}(u(q)-v(q))=0}$ for $ \xi \in \partial\Omega
   $ and $(\Delta v)^{m}\wedge \beta^{n-m} = (\Delta u)^{m}\wedge \beta^{n-m}\;\; in\; \Omega.$  Then $u= v$ in $\Omega.$
\end{cor}
\begin{proof}
Apply Corollary \ref{cor2} to get the result.
\end{proof}
\begin{pro}\label{pro6}
 For $u,v \in \mathcal{SH}_{m}\cap L_{loc}^{\infty}$, we have
 $$(\Delta \max \lbrace u,v \rbrace )^{m}\wedge \beta^{n-m}\geq \chi_{\lbrace u>v \rbrace}(\Delta u)^{m}\wedge \beta^{n-m}+ \chi_{\lbrace u\leq v \rbrace}(\Delta v)^{m}\wedge \beta^{n-m} $$ where $\chi_{A}$ is the characteristic function of a set $A.$
 \end{pro}
\begin{proof}
By changing the roles of $u$ and $v$, it suffices to prove that
$$\displaystyle{\int_{K}(\Delta(\max\lbrace u, v\rbrace)^{m}\wedge \beta^{n-m}} \geq \displaystyle{\int_{K}(\Delta u)^{m}\wedge \beta^{n-m}}$$
for every compact set $K \subseteq \{u \geq v\}$. Since $ u, v$ are bounded, we may
assume that $-1 \leq u, v \leq 0 $ and $-1 \leq u_{\varepsilon}, v_{\varepsilon} \leq 0,$ where $u_{\varepsilon} := u \ast \rho_{\varepsilon}$ is the
regularization of $u$. By Theorem \ref{th1} , we can assume that $G \subseteq \Omega$ is an open set
of small capacity such that $u, v$ are continuous on $\Omega \setminus G$. Then $u_{\varepsilon},v_{\varepsilon}$ converge
uniformly to $u, v$ on $ \Omega \setminus G$ respectively as $ \varepsilon$ tends to $0$. For any $\delta > 0$, we can
find an arbitrarily small neighborhood $L \;of\; K$ such that $u_{\varepsilon} > v_{\varepsilon} -\delta$ on $L\setminus G$
for $\varepsilon$ sufficiently small. Applying Proposition \ref{pro5}, we get $(\Delta u_{\varepsilon} )^{m}\wedge \beta^{n-m}$ converges weakly to $(\Delta u)^{m}\wedge \beta^{n-m}$,
and
\begin{equation*}
\begin{array}{ll}
 \displaystyle{\int_{K}(\Delta u)^{m}\wedge \beta^{n-m}}
&\leq \displaystyle{\liminf_{\varepsilon\longrightarrow 0 } \int_{L}(\Delta u_{\varepsilon})^{m}\wedge \beta^{n-m}}\\
& \leq \displaystyle{\liminf_{\varepsilon\longrightarrow 0}  \int_{G}(\Delta u_{\varepsilon})^{m}\wedge \beta^{n-m}+ \int_{L \setminus G}(\Delta u_{\varepsilon})^{m}\wedge \beta^{n-m}}\\
& \leq C(G,\Omega)+\displaystyle{\liminf_{\varepsilon\longrightarrow 0}  \int_{L \setminus G}(\Delta\max\lbrace u_{\varepsilon}+\delta,v_{\varepsilon}\rbrace)^{m}\wedge \beta^{n-m}}\\
&=C(G,\Omega)+\displaystyle{ \int_{L \setminus G}(\Delta\max\lbrace u+\delta,v\rbrace)^{m}\wedge \beta^{n-m}}.
\end{array}
\end{equation*}
The third inequality above follows from the definition of capacity and the
fact that $\max\lbrace u_{\varepsilon}+\delta, v_{\varepsilon}\rbrace =u_{\varepsilon}+\delta$ on a neighborhood of $L\setminus G$.
By taking $L$ very close to $K$ and $ C(G,\Omega)$ arbitrarily small, we have$\displaystyle{\int_{K}(\Delta\max\lbrace u+\delta, v\rbrace)^{m}\wedge \beta^{n-m}} \geq \displaystyle{\int_{K}(\Delta u)^{m}\wedge \beta^{n-m}}.$
 Let $\delta\longrightarrow 0 $ to get the above inequality .
\end{proof}
\section{ The Dirichlet problem}\label{sec3}
In this section, we will prove the first and the second order estimates. This will allows us  to solve the Dirichlet problem for quaternionic Hessian operator and prove the equivalent characterization of the maximality.
\subsection{ A priori estimates}

 Our goal is to prove the estimate $\Vert u \Vert_{C^{1,1}(\overline{B})}\leq C $, C depends only on $\varphi.$ We use
  some ideas from \cite{CKNS, CNS, T}  . We use the
notation $ u_{l}=\dfrac{\partial u}{\partial q_{l}},\;\;u_{\overline{l}}=\dfrac{\partial u}{\partial \overline{q}_{l}},$ where $$\frac{\partial u}{\partial \overline{q}_{l}}=\frac{\partial u}{\partial x_{4l}}+\mathrm{i}\frac{\partial u}{\partial x_{4l+1}}+\mathrm{j}\frac{\partial u}{\partial x_{4l+2}}+\mathrm{k}\frac{\partial u}{\partial x_{4l+3}},\;\;\;\; \frac{\partial u}{\partial q_{l}}=\frac{\partial u}{\partial x_{4l}}-\mathrm{i}\frac{\partial u}{\partial x_{4l+1}}-\mathrm{j}\frac{\partial u}{\partial x_{4l+2}}-\mathrm{k}\frac{\partial u}{\partial x_{4l+3}} $$ and $$ q_{l}=x_{4l}+\mathrm{i}x_{4l+1}+\mathrm{j}x_{4l+2}+\mathrm{k}x_{4l+3},\;\;l=0,\ldots,n-1.$$
 The real partial derivatives of $u$ will
be denoted by $u_{x_{4l}},u_{x_{4l+1}},u_{x_{4l+2}} \mbox{ and } u_{x_{4l+3}},\mbox{ and  
by } u_{\alpha}$  the real derivative of $u$ in direction $\alpha$ where $\alpha \in \{0,\ldots,4n-1\}$.
 Without loss of generality, we consider the equation
\begin{equation}\label{27}
F_{m}[u]:=\tilde{S}_{m}((u_{l\overline{s}}))=1\;\;\;\mbox{where}\;\; u_{l\overline{s}}=\frac{\partial^{2}u}{\partial \overline{q}_{s}\partial q_{l}}.
\end{equation}
By section \ref{sec2}, we have $(a^{l\overline{s}}) > 0$ and since $(u_{l\overline{s}})$ is diagonal then  $(a^{l\overline{s}})$ is diagonal. Thus the product of these is hyperhermitian matrix.
Computing the derivative of both sides of (\ref{27}) with respect to a real partial derivative at a direction $\alpha$ we get

\begin{equation}\label{28}
L(u_{\alpha })=0,\;\;\;\mbox{where}\;\;L:= a^{l\overline{s}}\partial_{l\overline{s}} \mbox{ is the linearized elliptic operator of}\; F_{m}, 
\end{equation}
 and $$(a^{l\overline{s}})=\mathcal{F}_{m}((u_{p\overline{r}})):=\Big(\frac{\partial\tilde{S}_{m} }{\partial u_{l\overline{s}}}((u_{p\overline{r}}))\Big).$$
Set
 \begin{equation}\label{29}
\begin{array}{ll}
 q_{l}u_{r}-\overline{q}_{r}u_{\overline{l}}
&:=\Big[\big( x_{4l}u_{x_{4r}} -x_{4r}u_{x_{4l}}\big) +\big(x_{4l+1}u_{x_{4r+1}}-x_{4r+1}u_{x_{4l+1}} \big )\\ &\;\;\;\; +\big( x_{4l+2}u_{x_{4r+2}}-x_{4r+2}u_{x_{4l+2}} \big )+\big(x_{4l+3}u_{x_{4r+3}}-x_{4r+3}u_{x_{4l+3}}\big)\Big]\\
&+i\Big[\big( x_{4l+1}u_{x_{4r}} -x_{4r}u_{x_{4l+1}}\big) +\big(x_{4l+3}u_{x_{4r+1}}-x_{4r+2}u_{x_{4l+3}} \big )\\ &\;\;\;\; -\big( x_{4l+2}u_{x_{4r+3}}-x_{4r+3}u_{x_{4l+2}} \big )-\big(x_{4l}u_{x_{4r+1}}-x_{4r+1}u_{x_{4l}}\big)\Big]\\
&+j \Big[\big( x_{4l+1}u_{x_{4r+3}} -x_{4r+3}u_{x_{4l+1}}\big) +\big(x_{4l+2}u_{x_{4r}}-x_{4r}u_{x_{4l+2}} \big )\\ &\;\;\;\; -\big( x_{4l+3}u_{x_{4r+1}}-x_{4r+1}u_{x_{4l+3}} \big )-\big(x_{4l}u_{x_{4r+2}}-x_{4r+2}u_{x_{4l}}\big)\Big]\\
&+k\Big[\big( x_{4l+2}u_{x_{4r+1}} -x_{4r+1}u_{x_{4l+2}}\big) +\big(x_{4l+3}u_{x_{4r}}-x_{4r}u_{x_{4l+3}} \big )\\ &\;\;\;\; -\big( x_{4l+1}u_{x_{4r+2}}-x_{4r+2}(u)_{x_{4l+1}} \big )-\big(x_{4l}u_{x_{4r+3}}-x_{4r+3}u_{x_{4l}}\big)\Big].
\end{array}
\end{equation}
 And let $ V $ be the associate vector field defined by:
$$\begin{array}{ll}
 V
&=\sum_{0 \leq l < r \leq n-1}  \Big[\Big(\big( x_{4l}\partial_{x_{4r}} -x_{4r}\partial_{x_{4l}}\big) +\big(x_{4l+1}\partial_{x_{4r+1}}-x_{4r+1}\partial_{x_{4l+1}} \big )\\ &\;\;\;\; +\big( x_{4l+2}\partial_{x_{4r+2}}-x_{4r+2}\partial_{x_{4l+2}} \big )+\big(x_{4l+3}\partial_{x_{4r+3}}-x_{4r+3}\partial_{x_{4l+3}}\big)\Big)\\
&+\mathrm{i}\Big(\big( x_{4l+1}\partial_{x_{4r}} -x_{4r}\partial_{x_{4l+1}}\big) +\big(x_{4l+3}\partial_{x_{4r+2}}-x_{4r+2}\partial_{x_{4l+3}} \big )\\ &\;\;\;\; -\big( x_{4l+2}\partial_{x_{4r+3}}-x_{4r+3}\partial_{x_{4l+2}} \big )-\big(x_{4l}\partial_{x_{4r+1}}-x_{4r+1}\partial_{x_{4l}}\big)\Big)\\
&+\mathrm{j} \Big(\big( x_{4l+1}\partial_{x_{4r+3}} -x_{4r+3}\partial_{x_{4l+1}}\big) +\big(x_{4l+2}\partial_{x_{4r}}-x_{4r}\partial_{x_{4l+2}} \big )\\ &\;\;\;\; -\big( x_{4l+3}\partial_{x_{4r+1}}-x_{4r+1}\partial_{x_{4l+3}} \big )-\big(x_{4l}\partial_{x_{4r+2}}-x_{4r+2}\partial_{x_{4l}}\big)\Big)\\
&+\mathrm{k}\Big(\big( x_{4l+2}\partial_{x_{4r+1}} -x_{4r+1}\partial_{x_{4l+2}}\big) +\big(x_{4l+3}\partial_{x_{4r}}-x_{4r}\partial_{x_{4l+3}} \big)\\ &\;\;\;\; -\big( x_{4l+1}\partial_{x_{4r+2}}-x_{4r+2}\partial_{x_{4l+1}} \big )-\big(x_{4l}\partial_{x_{4r+3}}-x_{4r+3}\partial_{x_{4l}}\big)\Big)\Big]
\end{array}
$$
Then, $ V $ is an infinitesimal generator of an element of $ \mathrm{Sp}(n) $. To see this, let $ M $ be the matrix defined as follows: For each pair $ (l, r) $ with $ 0 \leq l < r \leq n-1 $, $ M $ contains a block:

$$
M = \bigoplus_{0 \leq l < r \leq n-1} \begin{pmatrix} 0 & B_{lr} \\ -B_{lr}^\ast & 0 \end{pmatrix},
$$

where $ B_{lr} = I_4 + \mathbf{i} I_i + \mathbf{j} I_j + \mathbf{k} I_k $. Here:
 $ I_4 $ is the $ 4 \times 4 $ identity matrix.
 $ I_i, I_j, I_k $ are $ 4 \times 4 $ matrices representing quaternionic multiplication by $ \mathbf{i}, \mathbf{j}, \mathbf{k} $, respectively:

$$
I_i = \begin{pmatrix} 0 & -1 & 0 & 0 \\ 1 & 0 & 0 & 0 \\ 0 & 0 & 0 & -1 \\ 0 & 0 & 1 & 0 \end{pmatrix}, \quad
I_j = \begin{pmatrix} 0 & 0 & -1 & 0 \\ 0 & 0 & 0 & 1 \\ 1 & 0 & 0 & 0 \\ 0 & -1 & 0 & 0 \end{pmatrix}, \quad
I_k = \begin{pmatrix} 0 & 0 & 0 & -1 \\ 0 & 0 & -1 & 0 \\ 0 & 1 & 0 & 0 \\ 1 & 0 & 0 & 0 \end{pmatrix}.
$$

We have $ M^\ast + M = 0 $, where $ M^\ast $ denotes the quaternionic conjugate transpose of $ M $. Hence, $ M \in \mathrm{sp}(n) $ (see  \cite[Proposition 4.44]{ Ba}). Furthermore, the vector field $ V $ is  generated by  $ M $ via the correspondence:

$$
V = \sum_{i=0}^{4n-1} (M \mathbf{x})i \partial{x_i}.
$$

 Then the claim is proven. Hence, by  \cite{CNS}, it follows that $$L\Big(V\Big)[u]=VF_m[u]=V(1)=0.$$
Since
$
 L\Big(q_{l}u_{r}-\overline{q}_{r}u_{\overline{l}}\Big)
=a^{p\overline{s}}[q_{l}u_{r}-\overline{q}_{r}u_{\overline{l}}]_{p\overline{s}}$, it follows from the above that
\begin{equation}\label{30}
a^{p\overline{s}}[q_{l}u_{r}-\overline{q}_{r}u_{\overline{l}}]_{p\overline{s}}=0.
\end{equation}
The concavity of $\tilde{S}_{m}^{1/m}$ on $\tilde{\Gamma}_{m}$ implies that  $G:=log\tilde{S}_{m}$ is concave. Since  $\dfrac{\partial^{2}}{\partial x_{\alpha}\partial q_{l}}=\dfrac{\partial^{2}}{\partial q_{l} \partial x_{\alpha}}$ and $\dfrac{\partial^{2}}{\partial x_{\alpha}\partial \overline{q_{l}}}=\dfrac{\partial^{2}}{\partial \overline{q_{l}}\partial x_{\alpha}}$, with $\alpha\in\{0,\dots,4n-1\}$,
if we take   the logarithms of the both sides of (\ref{27}) and differentiate twice  with respect to $\frac{\partial}{\partial x_\alpha}$  we find first
$$\sum_{p,s}\dfrac{\partial G}{\partial u_{p\overline{s}}}u_{\alpha p\overline{s}}=0\;\;\; \mbox{ where }\;\;\; u_\alpha=\frac{\partial u}{\partial x_\alpha}.$$
and hence
$$\sum_{p,s,l,r}\dfrac{\partial^{2}G}{\partial u_{p\overline{s}}\partial u_{l\overline{r}}}u_{\alpha p\overline{s}} u_{\alpha l\overline{r}}+\sum_{p,s}\dfrac{\partial G}{\partial u_{p\overline{s}}}u_{\alpha \alpha p\overline{s}}=0 \mbox{ where } \;\;u_{\alpha \alpha}:= \frac{\partial^2 u}{\partial x^2_\alpha}.$$ 
Since $G$ is concave, then  the first term of the above equality is nonpositive and we get
\begin{equation}\label{31}
a^{p\overline{s}}u_{\alpha \alpha p\overline{s}}\geq 0.
\end{equation}
Let $B=B(0,1)$ be the unit ball in $\mathbb{H}^{n}$ and  $\varphi \in C^{\infty}(\overline{B})$ be harmonic in $B$. Let $C$  a different constant depending only on $\Vert \varphi\Vert_{C^{3,1}(\overline{B})}$
Set $\psi(q):=(\Vert q \Vert^{2}-1)/2$. Apply the comparison principal, we get for sufficiently big $C,\;\;
\varphi+C\psi \leq u \leq \varphi .$ From this and (\ref{27}) follow
\begin{equation}\label{32}
\Vert u \Vert_{C^{0,1}( \overline{B})} \leq C.
\end{equation}
Fix a coordinate system $(q_{0},\ldots,q_{n-1})\in \mathbb{H}^{n}$.
$$\begin{array}{ll}
 &\mbox{Let}\;\; t_{\alpha}= x_{4\alpha+3}\;\;\mbox{if}\;\;0\leq \alpha\leq n-1\;\;\mbox{and}\;\;
 t_{\alpha+n}= x_{4\alpha+2}\;\;\mbox{if}\;\;0\leq \alpha \leq n-1,\\
& \mbox{let} \;\;t_{\alpha+2n}= x_{4\alpha+1}\;\;\mbox{if}\;\;0\leq \alpha \leq n-1
 \;\;\mbox{;}\;\;t_{\alpha+3n}= x_{4\alpha}\;\;\mbox{if}\;\;\;0\leq \alpha \leq n-2,\\
 &\mbox{and}\;\; N = x_{4n-4}.\\
\end{array}$$

Now we will  estimate the double tangentiel derivative $u_{t_{\alpha}t_{\beta}}$ on $\partial B$.\\ For any point $\zeta$ in boundary of $B$, we may suppose
that $\zeta=(0,\ldots,0,1)$. Denote $t_{\alpha}, t_{\beta}$ the (real) tangential directions at $ \zeta $
  and  $ N $ the outer normal direction $x_{4n-4}$ axis.
 We set $$\displaystyle{x_{4n-4}= t_{4n-1}=\varrho(t^{'})=-\dfrac{1}{2}\sum_{k=0}^{4n-2}t_{k}^{2}+O(\vert t^{'}\vert ^{3})}\;\;\;\; where \;\;t^{'}=(t_{0},\ldots,t_{4n-2}).$$
We have $ u=\varphi$ on $\partial B$, then
$$(u-\varphi)(t^{'},\varrho(t^{'}))=0\;\;\; \mbox{where} \;\;t=(t^{'},t_{4n-1}).$$ By differentiation, at $\zeta$ we get
$$(u-\varphi)_{t_{\alpha}t_{\beta}}=-(u-\varphi)_{N}\varrho_{t_{\alpha}t_{\beta}}=(u-\varphi)_{N}\delta_{\alpha \beta} \;\;\mbox{for}\;\; t_{\alpha},t_{\beta}\;\;\; 0\leq \alpha,\beta\leq n-2,$$
thus
\begin{equation}\label{33}
u_{t_{\alpha}t_{\beta}} = \varphi_{t_{\alpha}t_{\beta}} + (u -\varphi)_{N}\delta_{\alpha \beta}.
\end{equation}
 From (\ref{32}) follows
 \begin{equation}\label{34}
 \vert u_{t_{\alpha}t_{\beta}}(\zeta)\vert \leq C\;\;;\zeta \in \partial B.
 \end{equation}

Next we estimate the mixed tangential-normal derivative $u_{tN}(\zeta_{0})$ for
a fixed $\zeta_{0} \in\partial B$.\\ As before we may assume that $\zeta_{0} =(0,\ldots,0,1)$, so that at $ \zeta_{0}$ we
have $N=\partial/\partial x_{4n-4}$. First assume that for some $0\leq l \leq n-2, \;\;t=\partial /\partial x_{4l}$ and set
$$\begin{array}{ll}
v &:= \dfrac{1}{2}Re\Big(q_{l}(u-\varphi)_{4n-4} - \overline{q}_{4n-4}(u-\varphi)_{\overline{l}}\Big)\\
&=\big( x_{4l}(u-\varphi)_{x_{4n-4}} -x_{4n-4}(u-\varphi)_{x_{4l}}\big) +\big(x_{4l+1}(u-\varphi)_{x_{4n-3}}-x_{4n-3}(u-\varphi)_{x_{4l+1}} \big )\\ &\;\;\;\; +\big( x_{4l+2}(u-\varphi)_{x_{4n-2}}-x_{4n-2}(u-\varphi)_{x_{4l+2}} \big )+\big(x_{4l+3}(u-\varphi)_{x_{4n-1}}-x_{4n-1}(u-\varphi)_{x_{4l+3}}\big).
\end{array}$$
Then $v=0$ on $\partial B,\;\;\vert v \vert \leq C $ on $\partial B(\zeta_{0},1)\cap \overline{B}$ and by (\ref{30})
$\pm a^{j\overline{k}}v_{j\overline{k}}\geq -C \displaystyle{\sum_{j}a^{j\overline{j}}}.$
Consider the barrier function $w := \pm v -C_{1}\vert q-\zeta_{0}\vert^{2} +C_{2}\psi,$ we
can choose constants $0\ll C_{1}\ll C_{2}$ under control so that $w\leq 0$ on
$ \partial( B \cap  B( \zeta_{0},1))$ and $a^{j\overline{k}}w_{j\overline{k}}\geq 0$ in $B \cap B( \zeta_{0},1).$ Therefore $w\leq 0$ in
 $B \cap B( \zeta_{0},1)$,
$$\vert v\vert\leq C_{1}\vert q-\zeta_{0}\vert^{2}-C_{2}\psi.$$
And it follows that $$\vert v_{x_{4n-4}}(\zeta_{0})\vert \leq C .$$ At $\zeta_{0}$ we have
$$v_{x_{4n-4}} = -(u-\varphi)_{x_{4l}} -(u-\varphi)_{x_{4l}x_{4n-4}},$$
and thus $$\vert u_{x_{4l}x_{4n-4}}(\zeta_{0})  \vert \leq C.$$

To estimate $u_{x_{4l+1}x_{4n-4}}$,
 we take the first imaginary component of vector part of
 \begin{equation}\label{35}
 q_{l}(u-\varphi)_{4n-4} - \overline{q}_{4n-4}(u-\varphi)_{\overline{l}}
 \end{equation}
  defined by:
  \begin{equation*}
   \begin{array}{ll}
&v: =\mathrm{i} \Big[\big( x_{4l+1}(u-\varphi)_{x_{4n-4}} -x_{4n-4}(u-\varphi)_{x_{4l+1}}\big) +\big(x_{4l+3}(u-\varphi)_{x_{4n-2}}-x_{4n-2}(u-\varphi)_{x_{4l+3}} \big)\\&\;\;\;\;\;\;-\big( x_{4l+2}(u-\varphi)_{x_{4n-1}}-x_{4n-1}(u-\varphi)_{x_{4l+2}} \big )-\big(x_{4l}(u-\varphi)_{x_{4n-3}}-x_{4n-3}(u-\varphi)_{x_{4l}}\big)\Big] .
\end{array}
  \end{equation*} and proceed similarly.

To estimate $u_{x_{4l+2}x_{4n-4}}$ and $u_{x_{4l+3}x_{4n-4}}$,
we choose respectively $v$ the second  and the third imaginary component of the vector part of (\ref{35}) and proceed similarly.

Let estimate $u_{x_{4n-3}x_{4n-4}},u_{x_{4n-2}x_{4n-4}}\;$ and $u_{x_{4n-1}x_{4n-4}}.$ Applying  
(\ref{30}) for $l=r=n-1$, we get $$ a^{p\overline{s}}\big[q_{n-1}u_{n-1}-\overline{q}_{n-1}u_{\overline{n-1}}\big]_{p\overline{s}}=0.$$ Then for $t=\partial/\partial x_{4n-3}$
 we consider
$$v^{i}= \mathrm{i}\Big[\big( x_{4n-3}(u-\varphi)_{x_{4n-4}} -x_{4n-4}(u-\varphi)_{x_{4n-3}}\big) +\big(x_{4n-1}(u-\varphi)_{x_{4n-2}}-x_{4n-2}(u-\varphi)_{x_{4n-1}}\big )\Big] ,$$
for $t=\partial/\partial x_{4n-2}$ and $t=\partial/\partial x_{4n-1}$ we consider
$$v^{j}= \mathrm{j}\Big[\big( x_{4n-2}(u-\varphi)_{x_{4n-4}} -x_{4n-4}(u-\varphi)_{x_{4n-2}}\big) +\big(x_{4n-3}(u-\varphi)_{x_{4n-1}}-x_{4n-1}(u-\varphi)_{x_{4n-3}}\big )\Big]  $$ and $$v^{k}= \mathrm{k}\big[( x_{4n-1}(u-\varphi)_{x_{4n-4}} -x_{4n-4}(u-\varphi)_{x_{4n-1}}) +(x_{4n-2}(u-\varphi)_{x_{4n-3}}-x_{4n-3}(u-\varphi)_{x_{4n-2}} )\big].$$
 We proceed similarly to
  obtain
\begin{equation}\label{36}
\vert u_{tN} (\zeta_{0})\vert \leq C.
\end{equation}

 Finally we turn to estimate the nomal-normal derivative
\begin{equation}\label{37}
\vert u_{n-1 \overline{n-1}}(\zeta_{0})\vert \leq C.
\end{equation}
Indeed, from (\ref{34}), (\ref{36}) and (\ref{31}) follow that all the
eigenvalues of the real Hessian matrix $D^{2}u$ are bounded from above by
$C$ in $B$. But since $u$ is  subharmonic, it implies that they must
 be bounded from below by $-(4n-1)C$. So it  remains to show (\ref{37}).
By (\ref{34}) and (\ref{36}) at $\zeta_{0}$ we may write
\begin{equation}\label{38}
1= u_{n-1\overline{n-1}}S_{m-1}^{'}(\zeta_{0})+O(1),
\end{equation}
where $S_{m-1}^{'}(\zeta_{0})=\tilde{S}_{m-1}((u_{p\overline{q}}(\zeta_{0}))^{'})$
 and $(u_{p\overline{q}})^{'}$
 denote the $(n - 1)\times(n - 1)$ matrix created by deleting the nth row
and nth column in $(u_{p\overline{q}}).$
 Following Blocki \cite{B} we consider  a smooth mapping\\
$$\Psi:\;\; (\overline{B}\cap\overline{B}(\zeta_{0},1))\times \mathbb{H}^{n}\longrightarrow \mathbb{H}^{n}$$  such that for every $q \in \overline{B}\cap\overline{B}(\zeta_{0},1)$ the mapping $\Psi_{q}= \Psi(q,.)$ is an
orthogonal isomorphism of $\mathbb{H}^{n}$, $\Psi_{\zeta}(\zeta_{0})=\zeta $ for $\zeta \in \partial B\cap \overline{B}(\zeta_{0},1)$
and $\Psi_{\zeta_{0}}$
 is the identity. We then have $$S^{'}_{m-1}(\zeta) = \tilde{S}_{m-1}(U(\zeta))\;\;\;for \;\;\zeta \in \partial B\cap \overline{B}(\zeta_{0},1), $$
 where by (\ref{33})  $$ U(\zeta)=R(\zeta)+ u_{N}(\zeta)I\;\;\mbox{ and}\;\;
R(\zeta):=\big((\varphi \circ \Psi_{\zeta})_{p\overline{q}}(\zeta_{0})\big)^{'}-\varphi_{N}(\zeta)I.$$
It is clear that $$ \big\Vert R \big\Vert_{C^{1,1}( \overline{B} \cap \overline{B}(\zeta_{0},1))} \leq C .$$
  Define the $(n-1)\times (n-1)$ positive definite matrix
$$\mathrm{G}_{0}:=\mathcal{F}_{m-1}(U(\zeta_{0}))= \Big(\dfrac{\partial \tilde{S}_{m-1}}{\partial u_{j\overline{k}}}(U(\zeta_{0}))\Big)\;\;\mbox{where}\;\; U(\zeta_{0})=\big((u_{p\overline{q}})(\zeta_{0})\big)^{'}.$$
As in \cite{T} we
may assume that  $S^{'}_{m-1}$, is minimized on $\partial B$  at $\zeta_{0},$
  then it follows from (\ref{81}) and (\ref{82}) that for $m \geq 2$
 \begin{equation*}
\begin{array}{ll}
tr( \mathrm{G}_{0}(U(\zeta)-U(\zeta_{0})))&=tr(\mathrm{G}_{0}(U(\zeta)))-tr(\mathrm{G}_{0}(U(\zeta_{0})))\\
&=tr\big[U(\zeta)\mathcal{F}_{m-1}(U(\zeta_{0}))\big]-tr\big[U(\zeta_{0})\mathcal{F}_{m-1}(U(\zeta_{0}))\big] \\
& \geq (m-1)\big[\tilde{S}_{m-1}(U(\zeta))\big]^{\frac{1}{m-1}} \big[\tilde{S}_{m-1}(U(\zeta_{0}))\big]^{\frac{m-2}{m-1}} -\tilde{S}_{m-1}(U(\zeta_{0})) \\
& \geq (m-1)\tilde{S}_{m-1}(U(\zeta_{0}))-\tilde{S}_{m-1}(U(\zeta_{0})) \\
& \geq (m-2)\tilde{S}_{m-1}(U(\zeta_{0}))\\
&\geq 0.
\end{array}
\end{equation*}
The case, $m=1$, is trivial.

 We will choose
$$v(\zeta):=  u_{x_{4n-4}}(\zeta)- u_{x_{4n-4}}(\zeta_{0})+ \langle D u(\zeta),\zeta-\zeta_{0}\rangle +(trG_{0})^{-1}tr[G_{o}(R(\zeta)-R(\zeta_{0})] \geq 0.$$
where $$Du=(u_{x_{0}},\ldots,u_{x_{4n-1}}),\; \zeta=(\zeta^{0},\ldots,\zeta^{4n-1})\;\mbox{and}\; \zeta_{0}=(0,\ldots,0,1,0,0,0).$$
Similarly as before, we define the barrier $w:= \pm v -C\vert q-\zeta_{0}\vert^{2} +C^{'}\psi.$
Differentiating equation (\ref{27}) we will find
$\vert a^{j\overline{k}}v_{j\overline{k}} \vert \leq C \displaystyle{\sum_{j}a^{j\overline{j}}}.$
Choosing $C\ll C^{'}$ under control, we get $w \leq 0 \;\; on\;\; \partial(B\cap B(\zeta_{0},1))$ and
$a^{j\overline{k}}w_{j\overline{k}} \geq 0$
 in $B\cap B(\zeta_{0},1) $.
 Therefore $w \leq 0$ in $B\cap B(\zeta_{0},1)$ and
$$u_{x_{4n-4}x_{4n-4}}(\zeta_{0})\leq C$$
hence by (\ref{36}) we obtain (\ref{37}).
Finally, we get
\begin{equation}\label{40}
 \Vert u \Vert_{C^{1,1}} (\overline{B}) \leq C
\end{equation}
\subsection{Main Theorem}
Now we are ready to prove our main theorem.
\begin{thm}\label{thbeforemain}
Let $B$ be a ball in $\mathbb{H}^{n}$, $\varphi$ continuous function on $\partial B$.\\ Then the following Dirichlet problem
\begin{equation*}
\left\{
     \begin{array}{lll}
       u\in\mathcal{SH}_{m}(B)\cap C(\overline{B}), \\
        (\Delta u)^m\wedge\beta^{n-m}=0,\\
        u=\varphi       \;\;\;\; \mbox{on} \;\;\;\; \partial B,
     \end{array}
   \right. \;\; \mbox{where} \;\;\; \beta=\dfrac{1}{8}\Delta(\parallel q \parallel^{2})
\end{equation*}
has a unique solution.
\end{thm}
 \begin{proof}
 Uniqueness is a consequence of Corollary \ref{cor2}. To show
the existence, first we assume that $\varphi$ is continuous and for a constant $a>0$
 we consider the Dirichlet problem
\begin{equation}\label{39}
\left\{
     \begin{array}{lll}
       u\in\mathcal{SH}_{m}(B)\cap C^{\infty}(B), \\
        (\Delta u)^m\wedge\beta^{n-m}=a\beta^{n},\\
        u=\varphi       \;\;\;\; \mbox{on} \;\;\;\; \partial B.
     \end{array}
   \right.\;\; \mbox{where} \;\;\; \beta=\dfrac{1}{8}\Delta(\parallel q \parallel^{2})
\end{equation}
By the Evans-Krylov theory see  \cite[Theorem 1]{CNS}, there exists a solution
of (\ref{39}) provided that we have an a priori bound
$\Vert u\Vert_{C^{1,1}(\overline{B})}\leq C$
where $C$ depends only on $a$ and $\varphi$.
Since we have proved  (\ref{40}), then  we can solve (\ref{00}). Let $\varphi$
be arbitrary continuous and approximate it from below by $\varphi_{j} \in C^{\infty}(\partial B).$ Let
$u_{j}$ be a solution of (\ref{39}) with $\varphi_{j}$ and $ a=\frac{1}{j}$. Let $\psi(q) =\vert q-q_{0}\vert^{2} -R^{2}$,
where $q_{0}$ is the center and $R$ the radius of $B$. For $k\geq j$ Corollary \ref{cor2}
gives
$u_{k} + j^{-1/ m}\psi -\Vert \varphi_{j}-\varphi \Vert _{L^{\infty}(\partial B)} \leq u_{j}\leq u_{k}$.
This implies that $u_{j}$ converges uniformly on $B$ to a certain $u$, which is a
solution by Proposition \ref{pro11}.
\end{proof}
\begin{rem} This result has been proven using other approaches and methods by  Harvey and Lawson (see \cite[Theorem 6.2]{HL1} and \cite[Theorem 14.8]{HL2})\footnote{We thank the authors for pointing this out to us and sending us their work.}.
\end{rem}
 \subsection{The maximality}
Now, we come to prove an equivalent characterization of the maximality in $\mathcal{SH}_{m}(\Omega),$ but first we recall the definition of maximal function.
\begin{defi}
Let $\Omega$ be an open set in $\mathbb{H}^{n}$, a function $u\in\mathcal{SH}_{m}(\Omega)$ is called maximal if \\ $v\in\mathcal{SH}_{m}(\Omega)$, $v\leq u$ outside a compact subset of $\Omega$ implies that $v\leq u$ in $\Omega.$
\end{defi}
In \cite{B}, Blocki proved that maximal functions $u\in\mathcal{SH}_{m}(\Omega)\cap C(\Omega)$ are precisely the solutions of the degenerate Dirichlet problem in the complex case.
Now we extend this characterization to  $\mathbb{H}^{n}.$
\begin{thm}
Let $\Omega$ be an open subset of $\mathbb{H}^{n}$ and $u\in\mathcal{SH}_{m}(\Omega)\cap C(\Omega).$ Then $u$ is maximal in $\Omega$ if and only if $(\Delta u)^{m}\wedge \beta^{n-m} =0.$
\end{thm}
\begin{proof}
 Corollary \ref{cor2} implies that if $u$
satisfies  $(\Delta u)^{m}\wedge \beta^{n-m} =0,$ then it is maximal. On the other hand, assume that
$u$ is maximal and let $B \Subset \Omega$ be a ball. By Theorem \ref{thbeforemain} there exists $u_{1}\in  C(\Omega)$
such that $u = u_{1}$ in $\Omega\setminus\ B\;,\; u_{1} \in \mathcal{SH}_{m}(B)$ and $(\Delta u_{1})^{m}\wedge \beta^{n-m} =0$ in $B$.
From the comparison principle (Corollary \ref{cor2}) follows $u_{1}\geq u$ in $B$, thus $u_{1} \in \mathcal{SH}_{m}(\Omega).$
Since $u$ is maximal, then by 6) in Proposition \ref{pro4}  $u_{1} = u$ and we get $(\Delta u)^{m}\wedge \beta^{n-m} =0.$
\end{proof}
{\bf Acknowledgements.} We would  like to thank the referees for their careful reading as
well as for their comments and suggestions that have contributed to improve
the readability and quality of the paper.


\begin{thebibliography}{99}
\bibitem[1]{A1}  Alesker, S.: \textit{Non-commutative linear algebra and plurisubharmonic functions of quaternionic variables}. Bull.Sci.Math,127,1-35(2003), http://dx.doi.org/10.1016/S0007-4497(02)00004-0.
\bibitem[2]{A2} Alesker, S.: \textit{Pluripotential theory on quaternionic manifolds}. J. Geom. Phys. 62 (2012), no. 5, 1189-1206,  http://dx.doi.org/
10.1016/j.geomphys.2011.12.001.
\bibitem[3]{Ba} Baker, A.: Matrix group. An introduction to Lie group theory. Springer Undergraduate Mathematics Series. Springer-Verlag London, Ltd, London, 2002. Xii+330 pp.
\bibitem[4]{B} Blocki, Z.: \textit{Weak solutions to the complex Hessian equation}. Ann. Inst. Fourier (Grenoble)
55 (2005), no. 5, 1735-1756.
\bibitem[5]{BE}  Bedford, E., Taylor, B.A.: \textit{A new capacity for plurisubharmonic functions}. Acta Math. 149
 (1–2) (1982) 1–40, http://
dx.doi.org/10.1007/BF02392348.
 \bibitem[6]{BO}  Boukhari, F.: \textit{H\"{o}lder continuous solutions to quaternionic Monge-Ampère
equations}. J. Math. Anal. Appl. 477 (2019), no. 1, 747--768, https://doi.org/10.1016/j.jmaa.2019.04.060.

\bibitem[7]{CHi}  Chinh, L.H.: Equation hessiennes complexes,  Th\`{e}se Doctorat. (2012) Toulouse.
\bibitem[8]{CKNS}  Caffarelli, L.,  Kohn, J.J.,  Nirenberg, L.,  Spruck, J.: \textit{The Dirichlet problem for nonlinear second-order elliptic equations. II. Complex Monge-Ampère, and uniformly elliptic}, equations. Comm. Pure Appl. Math. 38 (1985), no. 2, 209--252.
\bibitem[9]{CNS}   Caffarelli, L.,  Nirenberg, L.,  Spruck, J.: \textit{The Dirichlet problem for nonlinear second-order elliptic equations. III. Functions of the eigenvalues of the Hessian}. Acta Math. 155 (1985), no. 3-4, 261--301
\bibitem[10]{DK}  Dinew, S.,  Ko\l{}odziej, S.: \textit{ A priori estimates for complex Hessian equations.} Anal. PDE 7 (2014), no. 1, 227--244.
\bibitem[11]{G}   G\"{a}rding, L.: \textit{An inequality for Hyperbolic Polynomials}. J. Math. Mech. 8 (1959) 957-965.

  \bibitem[12]{J} Jo\~{a}o Pedro Morais, Svetlin Georgiev, Wolfgang Spr\"{o}ßig.: Real Quaternionic Calculus
Handbook. Real quaternionic calculus handbook. Birkhäuser/Springer, Basel, 2014. xii+216 pp.
\bibitem[13]{HL1} Harvey, F. Reese, Lawson, H. Blaine, Jr.: \textit{Dirichlet duality and the nonlinear Dirichlet problem.} Comm. Pure Appl. Math. 62 (2009), no. 3, 396--443.
    \bibitem[14]{HL2} Harvey, F. Reese, Lawson, H. Blaine, Jr.: \textit{Dirichlet duality and the nonlinear Dirichlet problem on Riemannian manifolds.} J. Differential Geom. 88 (2011), no. 3, 395--482.

\bibitem[15]{Li}  Li, S.-Y.: \textit{On the Dirichlet problems for symmetric function equations of the eigenvalues of the complex Hessian. Asian J. Math.} 8 (2004), no. 1, 87--106.
\bibitem[16]{LW} Liu, S., Wang, W.: On Pluripotential Theory Associated to Quaternionic m-Subharmonic Functions. J Geom Anal 33, 143 (2023). https://doi.org/10.1007/s12220-023-01197-x.
    \bibitem[17]{KS} Ko\l{}odziej, S., Sroka, M.: \textit{Regularity of solutions to the quaternionic Monge-Ampère equation.} J. Geom. Anal. 30 (2020), no. 3, 2852--2864. 
\bibitem[18]{NG}  Nguyen, N.-C.: On the H\"{o}lder continuous subsolution problem for the complex Monge Amp\`{e}re equation, Calc. Var. Partial Differential Equations 57 (2018), no. 1, Art. 8, 15
pp.
\bibitem[19]{SA}  Sadullaev, A., Abdullaev, B.: Potential theory in the class of m-subharmonic
functions, Proceeding of the Steklov Institute of Mathematics. 279 (2012), 155-
180.
\bibitem[20]{S} Sroka, M.: \textit{Weak solutions to the quaternionic Monge-Ampère equation}. Anal. PDE 13 (2020), no. 6, 1755--1776. 
\bibitem[21]{T}  Trudinger, N.S.: \textit{On the Dirichlet problem for Hessian equations}. Acta Math.
175 (1995), 151-164.
   \bibitem[22]{W1}  Wan, D.,  Zhang, W.: \textit{Quasicontinuity and maximality of quaternionic plurisubharmonic functions}. J. Math. Anal. Appl. 424 (2015) 86-103, https:// doi:10.1016/j.matpur.2012.10.002.
        \bibitem[23]{W2} Wan, D., Wang, W.: \textit{On quaternonic Monge Ampère operator, closed positive currents and
Lelong-Jensen type formula on quaternionic space}. Bull. Sci. math. 141 (2017) 267–311.
\bibitem[24]{W3}  Wang, W.:  \textit{On the optimal control method in quaternionic analysis}. Bull. Sci. Math. 135(8):988-
1010,2011, http://
dx.doi.org/10.1016/j.bulsci.2011.09.004.

\bibitem[25]{Z}  Zhu, J.: Dirichlet problem of quaternionic Monge-Amp\`{e}re equations. Isr. J. Math. 214 (2016), 597–619.
\end{thebibliography}
\end{document}